\documentclass[11pt,reqno]{amsart}
\usepackage{amssymb}
\usepackage[all]{xy}
\usepackage{fancyhdr}
\usepackage{hyperref}
\usepackage{tikz-cd}
\usepackage{cite}
\usepackage{amsmath,amsfonts}
\usepackage{graphicx}
\usepackage{wrapfig}
\usepackage{array}
\usepackage{amsthm}
\usepackage[left=1.2in, right=1.2in, top=1in, bottom=1in]{geometry}
\usepackage{etoolbox}

\patchcmd{\section}{\scshape}{\bfseries}{}{}
\makeatletter
\renewcommand{\@secnumfont}{\bfseries}
\makeatother
\usetikzlibrary{matrix,arrows,decorations.pathmorphing}
\usetikzlibrary{arrows,positioning}   

\DeclareMathOperator{\Hom}{Hom}
\DeclareMathOperator{\Spec}{Spec}
\DeclareMathOperator{\Frac}{Frac}
\DeclareMathOperator{\Img}{Img}
\DeclareMathOperator{\Ker}{Ker}

\input xy
\xyoption{all}
\usepackage{secdot}  

\theoremstyle{plain}
\newtheorem{mydef}{\textbf{Definition}}[section]
\newtheorem{myeg}[mydef]{\textbf{Example}}
\newtheorem{mythm}[mydef]{\textbf{Theorem}}
\newtheorem*{nothma}{\textbf{Theorem A}}
\newtheorem*{nothmb}{\textbf{Theorem B}}
\newtheorem*{nothmc}{\textbf{Theorem C}}
\newtheorem*{nothmd}{\textbf{Theorem D}}
\newtheorem*{nothme}{\textbf{Theorem E}}
\newtheorem*{nothmf}{\textbf{Theorem F}}
\newtheorem*{nothmg}{\textbf{Theorem G}}

\newtheorem{rmk}[mydef]{\textbf{Remark}}
\newtheorem{lem}[mydef]{\textbf{Lemma}}
\newtheorem{pro}[mydef]{\textbf{Proposition}}

\newtheorem*{que}{\textbf{Question}}
\newtheorem*{normk}{\textbf{Remark}}
\patchcmd{\abstract}{\scshape\abstractname}{\normalsize{\textbf{\abstractname}}}{}{}
\begin{document}

\title{Algebraic geometry over hyperrings}


\author{Jaiung Jun}
\address{Department of Mathematical Sciences, Binghamton University, Binghamton, NY 13902, USA}
\curraddr{}
\email{jjun@math.binghamton.edu}


\subjclass[2010]{14A99(primary), 14T99(secondary)}

\keywords{hyperring scheme, hyperring, hypergroup, tropical variety, $\mathbb{F}_1$-geometry}

\date{}

\dedicatory{}

\begin{abstract}
\normalsize{\noindent We develop basic notions and methods of algebraic geometry over the algebraic objects called hyperrings. Roughly speaking, hyperrings generalize rings in such a way that an addition is `multi-valued'. This paper largely consisits of two parts; algebraic aspects and geometric aspects of hyperrings. We first investigate several technical algebraic properties of a hyperring. In the second part, we begin by giving another interpretation of a tropical variety as an algebraic set over the hyperfield which canonically arises from a totally ordered semifield. Then we define a notion of an integral hyperring scheme $(X,\mathcal{O}_X)$ and prove that $\Gamma(X,\mathcal{O}_X)\simeq R$ for any integral affine hyperring scheme $X=\Spec R$. }
\end{abstract}

\maketitle
\tableofcontents

\section{Introduction}
An idea of algebraic objects with a multi-valued operation has been first considered by F.Marty in 1934 (see \cite{marty1935role}). Marty introduced the notion of hypergroups which generalizes groups by stating that an addition is multi-valued. Later, in 1956, M.Krasner introduced the concept of hyperrings in \cite{krasner1956approximation} to use them as a technical tool on the approximation of valued fields. Since its first appearance, very few algebro-geometric perspective of hyperrings has been studied (cf. \cite{Dav1}, \cite{Rota}). However, lately there has been considerable attention drawn to the theory of hyperrings under various motivations.\\  
In \cite{con3}, A.Connes and C.Consani noticed  that Connes' ad\`{e}le class space of a global field has a hyperring structure and initiated the study of algebraic geometry based on hyperrings. Connes and Consani also provided several evidences which show that hyperrings are algebraic structures which naturally appear in relation to number theory. For example, in their archimedean analogue of the p-typical Witt construction, the role of hyperstructures is crucial (see \cite{con6}). Furthermore, by viewing an underlying space of an affine algebraic group scheme $X=\Spec A$ as a set of `$\mathbf{K}$-rational points' for the Krasner's hyperfield $\mathbf{K}$ (Example \ref{krasner}), Connes and Consani showed that $X$ is a hypergroup when $X=\mathbb{G}_m$ and $\mathbb{G}_a$ (cf. \cite{con4}). This result has been partially generalized by the author in \cite{jun2015hyperstructures}. The author also used a notion of hyperrings to generalize the classical notion of valuations in \cite{jun2015valuations}.
\begin{rmk}\label{remark1}
Hyperrings were first considered by Connes and Consani as a possible replacement for monoids on the developments of $\mathbb{F}_1$-geometry. One of main goals in $\mathbb{F}_1$-geometry is to enlarge the category of schemes to realize $C:=\overline{\Spec \mathbb{Z}}$ as `a curve over $\mathbb{F}_1$' in such a way that a (suitably defined) Hasse-Weil zeta function of $C$ over $\mathbb{F}_1$ becomes the complete Riemann zeta function. A first natural candidate was the category of monoids (to enlarge the category of schemes). However, a generalized scheme theory over monoids is essentially toric (cf. \cite{deitmar2008f1}) whereas one expects that $\overline{\Spec \mathbb{Z}}$ has an infinite genus (cf. \cite{con1}, \cite{manin1995lectures}). This motivated Connes and Consani to search possible other candidates rather than monoids and the study of algebraic geometry over more general algebraic objects by many others (cf. \cite{con2}, \cite{Deitmar}, \cite{oliver1}, \cite{pena2009mapping}, \cite{soule2004varietes}).\\ In particular, B.To\"{e}n and M.Vaqui\'{e} first introduced algebraic geometry over a closed monoidal category by generalizing the functorial definition of schemes (see, \cite{toen2009dessous}). For example, the classical scheme theory and monoid schemes as in A.Deitmar \cite{Deitmar} agree with this picture. We further note that, in \cite{noah}, J.Giansiracusa and N.Giansiracusa showed that a scheme theory over semirings can be beautifully linked to tropical geometry by means of the theory of monoid schemes. However, since the category of hyperrings is not a monoidal category, our construction of hyperring schemes does not fit with the above generalized scheme theory. 
\end{rmk}
Besides from the aforementioned Connes and Consani's approaches, M.Marshall and P.G{\l}adki used multirings (which are more general objects than hyperrings) in relation to quadratic forms, Artin-Schreier theory, and real algebraic geometry (cf. \cite{mars2}, \cite{mars1}). J.Tolliver Studied a relation between Krasner's hyperfield approach and P.Deligne's approach \cite{deligne1984corps} to limits of local fields. Also, O.Viro implemented a notion of hyperfield to search for a firm algebraic foundation of tropical geometry (cf. \cite{viro}, \cite{viro2}). Note that supertropical algebras by Z.Izhakian and L.Rowen (cf. \cite{izhakian2014layered}, \cite{iza}) and blueprints by O.Lorscheid (cf. \cite{oliver1}, \cite{lorscheid2015scheme}) provide natural algebraic framework for tropical geometry and have certain connections to hyperrings. Lastly, we note that Connes and Consani proved that some class of hyperrings can be realized as an $\mathfrak{s}$-algebra (an algebra over the sphere spectrum) in their recent paper \cite{connesabsolute}.\\

The following is the organization of the paper:\\
\noindent In \S 2, we provide basic definitions and properties of hyperrings which will be used in sequel. In \S 3, we investigate algebraic properties of hyperrings. We first start by resolving an issue on the construction of quotients of hyperrings. A quotient of a hyperring by a (hyper) ideal has been known only for a special class of (hyper) ideals called normal (hyper) ideals. We first prove that in fact such construction is valid for any (hyper) ideal as follows:
\begin{nothma}$($Proposition \ref{quotienthyperring}$)$
Let $R$ be a hyperring and $I$ be a (hyper) ideal of $R$. Then the set $R/I$ of cosets has a canonical hyperring structure which generalizes the classical case.  
\end{nothma}
The classical one-to-one correspondence between ideals and congruence relations is no longer valid in the category of commutative semirings (Example \ref{congruenceexample}). The difficulty stems from the weakness of commutative semirings; lacking additive inverses. One may fix a such problem at a price; by allowing an addition to be `multi-valued', one obtains additive inverses. The following theorem states that such a one-to-one correspondence is still valid for hyperrings. 
\begin{nothmb}$(\S \ref{onetoone})$
Let $R$ be a hyperring. Then there is a one-to-one correspondence between the set of (hyper) ideals of $R$ and the set of congruence relations on $R$.
\end{nothmb}
This provides some evidence that hyperrings may be better suited to build an algebraic foundation of tropical geometry.\\
In \S 3.3, we prove several propositions which are analogues to the classical propositions. These results will be used in \S \ref{schmeandzeta} to construct an integral hyperring scheme. In particular, we prove the following version of Hilbert's Nullstellensatz of hyperrings.

\begin{nothmc}$($Lemma \ref{nilradical}$)$
Let $R$ be a hyperring and $I$ be a hyperideal of $R$. Then the following set
\[\sqrt{I}:=\{ r\in R \mid \exists n \in \mathbb{N} \textrm{ such that  } r^n \in I\}\]
is the intersection of all prime hyperideals of $R$ which contain $I$.
\end{nothmc}

In \S 4, we study geometric aspects of hyperrings. In the first subsection, we reformulate a tropical variety as a `positive part' of an algebraic set over a hyperfield. This work has been motivated by Viro's paper \cite{viro}. Viro manifested that an algebraic foundation of tropical geometry should be built on the framework of hyperfields and used some hyperfields to realize his goal. For the Viro's hyperfield approach over a semiring approach, see \cite[\S 6]{brugalle2014bit}. The novelty of our approach lies in the use of symmetrization process which is recently introduced by S.Henry (cf. \cite{Henry}). Briefly speaking, a symmetrization process `glues' two semigroups to obtain a hypergroup. In \cite{jaiungthesis}, we showed that such process can be generalized to a hyperring and a semiring in some cases (also see Appendix \ref{sym}). In particular, one always derives a hyperfield by patching two copies of a totally ordered idempotent semifield. By appealing to these results, we may expect to benefit from advantages of both semirings and hyperrings via a symmetrization. More precisely, after introducing a notion of an algebraic set over hyperrings, we prove the following:
\begin{nothmd}$($Proposition \ref{tropicalvareityispositivepart}$)$
Let $\mathbf{R}:=(\mathbb{R}_{max})_S$ be the hyperfield symmetrizing the tropical semifield $\mathbb{R}_{max}=(\mathbb{R}\cup \{-\infty\},+,\max\}$ and $s^n$ be the coordinate-wise symmetrization map. Let $X$ be a tropical variety over $\mathbb{R}_{max}$ in $(\mathbb{R}_{max})^n$. Then, there exist a (suitably defined) algebraic set $X_S$ over $\mathbf{R}$ and the following set bijection:
\[\varphi: X \simeq  (\Img(s^n)\cap X_S).  \]
\end{nothmd}
\noindent In the second subsection, we take a scheme-theoretic point of view and construct an integral hyperring scheme generalizing the classical notion of an integral scheme. One of the most important results in classical scheme theory is that any commutative ring $A$ is isomorphic to the ring of global regular functions on the topological space $X=\Spec A$. When one enlarges the category of commutative rings to the category of commutative monoids or commutative semirings, one still obtains the similar result. In other words, a commutative monoid (resp. a commutative semiring) can be realized as the commutative monoid (resp. the commutative semiring) of global regular functions on some topological space (see the references in Remark \ref{remark1}). Also, when $A$ is a $\mathbb{C}^*$-algebra (not necessarily commutative), there is an analogue known as Gelfand-Naimark-Segal construction. Therefore one is induced to ask the following question.
\begin{que}
Can a hyperring $R$ be realized as the hyperring of global regular functions on the topological space $X=\Spec R$?
\end{que}

We prove that the answer is affirmative when $R$ is a hyperring without any (multiplicative) zero-divisors as follows:
\begin{nothme}$($Theorem \ref{sheaf}, Proposition \ref{fullyfaithful}$)$
Let $R$ be a hyperring with no (multiplicative) zero-divisor. Then, for the hyperring scheme $(X=\Spec R, \mathcal{O}_X)$, we have $\Gamma(X,\mathcal{O}_X) \simeq R$. Furthermore, for each $\mathfrak{p} \in X$, the stalk $\mathcal{O}_{X,\mathfrak{p}}$ of $\mathcal{O}_X$ at $\mathfrak{p}$ exists and is isomorphic to a (suitably defined) localized structure $R_\mathfrak{p}$. In particular, the opposite category of hyperrings without zero-divisors is equivalent to the category of integral affine hyperring schemes.
\end{nothme}
\begin{normk}
Note that a similar construction has been done in \cite{Rota}. However, our goal is to retrieve the classical important result: $\Gamma(\Spec A, \mathcal{O}_{\Spec A})\simeq A$ for hyperrings whereas the authors of \cite{Rota} did not consider it.
\end{normk}

We also construct a possible replacement $\mathcal{F}_X$ of the structure sheaf $\mathcal{O}_X$ and obtain the following theorem. 

\begin{nothmf}$($Remark \ref{OFsame}, Propositions \ref{sheaffirre}, \ref{Fgen}$)$
Let $R$ be a hyperring and $X=\Spec R$ be a topological space of prime spectra. We associate a presheaf $\mathcal{F}_X$ of hyperrings on $X$ which satisfies the following: 
\begin{enumerate}
\item
If $R$ is a hyperring with no zero-divisors then $\mathcal{F}_X$ is a sheaf of hyperrings isomorphic to $\mathcal{O}_X$.
\item
If $X$ is irreducible, then $\mathcal{F}_X$ is a sheaf of hyperrings and the canonical map $R \longrightarrow \mathcal{F}_X(X)$ is strict and injective homomorphism of hyperrings. 
\item
If $X$ is irreducible and $R$ is local, then $R \simeq \mathcal{F}_X(X)$.
\end{enumerate}
\end{nothmf}

The remaining part of the paper is devoted to explain how the theory of hyperring schemes can be related with the classical scheme theory. It follows from the result (Proposition \ref{connect}) of Connes and Consani, one can naturally associate a hyperring to a commutative ring. By using this, we prove the following:

\begin{nothmg}$($Proposition \ref{qsheaflem}$)$
Let $A$ be an integral domain containing the field $\mathbb{Q}$ of rational numbers. Let $R:=A/\mathbb{Q}^{\times}$, $X=(\Spec A,\mathcal{O}_X)$, $Y=(\Spec R,\mathcal{O}_Y)$, and $\pi:A \longrightarrow A/\mathbb{Q}^{\times}$ be the canonical projection map. Then, as topological spaces, $Y$ is homeomorphic to $X$. Furthermore, for an open subset $U \subseteq X$, the hyperring of sections $\mathcal{O}_Y(U)$ is obtained from the ring of sections $\mathcal{O}_X(U)$ by the quotient construction of Connes and Consani as in Proposition \ref{connect}.
\end{nothmg}

\textbf{Acknowledgment}
\vspace{0.1cm}

This is a part of the author's Ph.D. thesis \cite{jaiungthesis}. The author thanks his academic advisor Caterina Consani. This paper has not been written without her support. The author also thanks Jeffrey Tolliver for proof-reading the early version of the paper and providing helpful comments. 

\section{Review: Hyperrings}
In this section, we quickly review the basic definitions and properties of hyperrings which will be used in this paper. For more details regarding theory of hyperrings, we refer the readers to \cite{corsini2003applications}.\\ 
In what follows, by a hyperoperation on a nonempty set $H$, we always mean a function \[+ : H \times H \rightarrow \mathcal{P}(H)^{*},\] where $\mathcal{P}(H)^*$ is the set of nonempty subsets of $H$. Note that the reason why we use $+$ notation is that we will only consider the commutative case throughout this paper. We also define, for any nonempty subsets $\forall A, B\subseteq H$, the following notation: 
\[A+B:=\bigcup_{a\in A, b\in B}(a+b).\] 
Therefore the notation $(x+y)+z$ makes sense as the set $A+\{z\}$ with $A:=x+y$. Lastly, for $x,y \in H$, when $x+y$ contains a single element $z$ we let $x+y=z$ for the notational simplicity. With these notations, for a subset $A\subseteq H$, and $a \in H$, by $a+A$ we mean that $\{a\}+A$.
\begin{mydef}
A canonical hypergroup $(H,+)$ is a nonempty set equipped with a hyperoperation $+$ with the following properties:
\begin{enumerate}
\item
$+$ is commutative;
$x+y=y+x\quad  \forall x,y \in H.$
\item
$+$ is associative;
$(x+y)+z=x+(y+z)\quad  \forall x, y, z\in H. $
\item
There is the identity element; $\exists!$ $0_H \in H$ such that $0_H+x=x=x+0_H\quad \forall x\in H.$
\item
There is a unique inverse; $\forall x\in H \quad \exists  !  y \in H \quad \text{such that} \quad 0\in x+y.$ We denote $y=-x$.
\item
Reversibility; $x \in y+z \iff z \in x-y$ $\forall x,y,z \in H$. 
\end{enumerate}
\end{mydef}

We note that hypergroups are more general objects than canonical hypergroups. For example, one does not assume the uniqueness of the identity or an inverse. However, in this paper we only consider canonical hypergroups. From now on, we will simply say hypergroups rather than canonical hypergroups for the sake of convenience.
\begin{mydef}
A Krasner hyperring $(R, + ,\cdot,0,1)$ is a nonempty set $R$ such that $(R,+,0)$ is a hypergroup and $(R,\cdot,1)$ is a commutative monoid which satisfy the following conditions:
 \begin{enumerate}
\item
$+$ and $\cdot$ are compatible; $ \forall x, y, z \in R$,  $ x(y+z)=xy+xz$,  $(x+y)z=xz+yz$.
\item
$0$ is an absorbing element; $\forall x \in R$,  $ x\cdot 0 = 0 = 0\cdot x$.
\item
$0 \neq 1$.
\end{enumerate}
When $(R\setminus \{ 0 \}, \cdot)$ is an abelian group, $(R, +, \cdot,0,1)$ is said to be a hyperfield.
\end{mydef}
Throughout the paper, all hyperrings are assumed to be Krasner hyperrings unless otherwise stated.
\begin{mydef}
Suppose that $(R_1, +_1, \cdot_1)$,  $(R_2, +_2, \cdot_2)$ are hyperrings. A map $\varphi : R_1 \longrightarrow R_2$ is said to be a homomorphism of hyperrings if
\begin{enumerate}
\item
$\varphi$ is a homomorphism of hypergroups; $\varphi(a+_1b) \subseteq \varphi(a)+_2\varphi(b) \quad \forall a, b \in R_1$.
\item
$\varphi$ is a homomorphism of monoids; $\varphi(a\cdot _1b) = \varphi(a)\cdot _2\varphi(b) \quad \forall a, b \in R_1$.
\item
$\varphi$ is said to be strict if $\varphi(a+_1b)= \varphi(a)+_2\varphi(b) \quad \forall a, b \in R_1$.
\end{enumerate}
\end{mydef}
\begin{mydef}\label{hyperringext}
Let $R$ be a hyperring. By a hyperring extension of $R$ we mean a hyperring $L$ such that there is an injective homomorphism $i:R \longrightarrow L$ of hyperrings. A sub-hyperring $H$ of $R$ is a subset of $R$ such that $H$ itself is a hyperring with the induced addition and multiplication.
\end{mydef}

\begin{myeg}(cf. \cite{con3})\label{krasner}
Let $\mathbf{K}:=\{0,1\}$ with the hyperoperation and the multiplication given in the following tables:
\[
\begin{tabular}{ | c | c | c | }
    \hline
    $+$ & $0$ & $1$  \\ \hline
    $0$ & $0$ & $1$  \\ \hline
    $1$ & $1$ & $\{0,1\}$  \\ 
   
    \hline
    \end{tabular} \qquad 
    \begin{tabular}{ | c | c | c | }
        \hline
        $\cdot$ & $0$ & $1$  \\ \hline
        $0$ & $0$ & $0$  \\ \hline
        $1$ & $0$ & $1$  \\ 
       \hline
        \end{tabular}
\]
\vspace{0.1cm}

One can easily check that $\mathbf{K}$ is a hyperfield. $\mathbf{K}$ is called the Krasner's hyperfield.
\end{myeg}

\begin{myeg}(cf. \cite{con3})\label{sign}
Let $\mathbf{S} = \{-1,0,1\}$. A hyperoperation and a multiplication are commutative and given by

\[
\begin{tabular}{ | c | c | c |c| }
    \hline
    $+$ &$-1$ & $0$ & $1$  \\ \hline
    $-1$ & $-1$ & $-1$ & $\{-1,0,1\}$ \\  \hline
    $0$ &$-1$ & $0$ & $1$  \\ \hline
    $1$ & $\{-1,0,1\}$ & $1$ & $1$ \\  
   
    \hline
    \end{tabular} \qquad 
    \begin{tabular}{ | c | c | c |c| }
        \hline
        $\cdot$ &$-1$ & $0$ & $1$  \\ \hline
        $-1$ & $1$ & $0$ & $-1$ \\  \hline
        $0$ &$0$ & $0$ & $0$  \\ \hline
        $1$ & $-1$ & $0$ & $1$ \\ 
       
        \hline
        \end{tabular}
\]
\vspace{0.1cm}

With the above operations, $\mathbf{S}$ becomes a hyperfield and is called the hyperfield of signs; one can observe that the hyperoperation follows the rule of signs.
\end{myeg}
\begin{myeg}(cf. \cite{viro})\label{viro}
Let $\mathcal{T}\mathbb{R}:=(\mathbb{R},+_{T},\cdot)$; the underlying set is the set of real numbers and the multiplication is the usual multiplication of real numbers. The (hyper)addition is given as follows:
\[
x+_{T}y = \left\{ \begin{array}{ll}
\{x\} & \textrm{if $|x|>|y|$}\\
\{y\} & \textrm{if $|x|<|y|$}\\
\{x\} & \textrm{if $x=y$}\\
\ [-|x|,|x|], \textrm{ a closed interval} & \textrm{if $x=-y$}\\
\end{array} \right.
\]
We call $\mathcal{T}\mathbb{R}$ the Viro's hyperfield; $\mathcal{T}\mathbb{R}$ will play a role of the tropical semifield in \S 4.1.
\end{myeg}
\begin{rmk}
The definition of a hyperring extension is subtle since any injective homomorphism of hyperrings does not have to be strict. However, in this paper, we stick with Definition \ref{hyperringext}. With this definition, one can easily check that $\mathcal{T}\mathbb{R}$ is a hyperfield extension of $\mathbf{S}$. In general, when $H$ is a sub-hyperring of a hyperring $R$, for any $a, b \in H$, we have $a+_Hb\subseteq a+_Rb$.  
\end{rmk}
\noindent Next, we review the notion of (prime) ideals for hyperrings and the basic properties of them. 

\begin{mydef}
Let $R$ be a hyperring. 
\begin{enumerate}
\item
A nonempty subset $I$ of $R$ is a hyperideal if: $\forall a, b \in I,\forall r \in R$, we have $a-r\cdot b \subseteq I$.
\item
A hyperideal $I \subsetneq R$ is prime if $I$ satisfies the following property: if $xy \in I$, then $x\in I$ or $y \in I$ $\forall x,y \in I$. 
\item
A hyperideal $I \subsetneq R$ is maximal if $I$ satisfies the following property: if $J \subsetneq R$ is a hyperideal of $R$ which contains $I$, then $I=J$.  
\end{enumerate}
\end{mydef}

\begin{pro}(cf. \cite[Proposition 2.12 and 2.13]{Dav1})
Let $R$ be hyperring. 
\begin{enumerate}
\item
Let $I$ be a proper hyperideal of $R$, i.e. $I \neq R$. Then there exists a maximal hyperideal $\mathfrak{m}$ such that $I \subseteq \mathfrak{m}$.
\item
Any maximal hyperideal $\mathfrak{m}$ is prime.
\end{enumerate}
\end{pro}

\begin{mydef}(cf. \cite[\S 2]{Rota})\label{zarisktopdef}
Let $R$ be a hyperring. We denote by $\Spec R$ the set of prime hyperideals of $R$. One can impose the Zariski topology on $\Spec R$ as in the classical case. In other words,
\begin{equation}\label{zariskitop}
\textrm{ a subset }A \subseteq \Spec R \textrm { is closed } \Longleftrightarrow A=V(I) \textrm{ for some hyperideal $I$ of } R, 
\end{equation}
where $V(I):=\{\mathfrak{p} \in \Spec R \mid I \subseteq \mathfrak{p}\}$.
\end{mydef}

The following proposition shows that Definition \ref{zarisktopdef} indeed makes sense. 
\begin{pro}(cf. \cite[\S 2]{Rota})\label{specRcomputation}
Let $R$ be a hyperring and $X=\Spec R$.
\begin{enumerate}
\item
Let $\{I_j\}_{j \in J}$ be a family of hyperideals of $R$. Then we have
\begin{equation}\label{V(I)compute}
\bigcap_{j \in J}V(I_j)=V(<\bigcup_{j \in J}I_j>), 
\end{equation}
where $<\bigcup_{j \in J}I_j>$ is the smallest hyperideal containing $\bigcup_{j \in J}I_j$. Note that such hyperideal exists since an arbitrary intersection of hyperideals is a hyperideal.
\item
Let $I$ and $I'$ be hyperideals of $R$, then we have
\begin{equation}\label{V(I')compute}
V(I)\bigcup V(I')=V(I \cap I').
\end{equation}
\end{enumerate}
\end{pro}
\noindent Next, we review the notion of a localization of a hyperring. This construction also has been promoted by R.Procesi-Ciampi and R.Rota in \cite{Rota}.\\ 
For a (multiplicative) submonoid $S$ of a hyperring $R$, one defines the localization $S^{-1}R$ as follows: as a set, $S^{-1}R$ is the set $(R \times S / \sim)$ of equivalence classes, where
\begin{equation}\label{eqrel}
(r_1,s_1)\sim(r_2,s_2) \Longleftrightarrow \exists x \in S \quad s.t. \quad xr_1s_2=xr_2s_1.
\end{equation}
Let $[(r,s)]$ be the equivalence class of $(r,s) \in R\times S$ under the equivalence relation \eqref{eqrel}. A hyperaddition of $S^{-1}R$ is given by
\[
[(r_1,s_1)] + [(r_2,s_2)] = [(r_1s_2+s_1r_2), s_1s_2] = \{ [(y, s_1s_2)] \mid y \in r_1s_2 + s_1r_2 \}.
\]
A multiplication is naturally given as follows:
\[
[(r_1,s_1)]\cdot [(r_2, s_2)] = [(r_1r_2, s_1s_2)].
\]
We denote by $\frac{r}{s}$ an element $[(r,s)]$. Note that as in the classical case, the localization map, $S^{-1}:R \longrightarrow S^{-1}R$ sending $r$ to $\frac{r}{1}$, is a homomorphism of hyperrings.

\begin{pro}(cf. \cite[\S 3]{Dav1})\label{localhyperring}
Let $R$ be a hyperring and $S$ be a (multiplicative) submonoid of $R$.
\begin{enumerate}
\item
For a hyperideal $I$ of $R$, the following set:
\[S^{-1}I:=\{\frac{i}{s}\mid i \in I, s \in S\}\]
is a hyperideal of $S^{-1}R$.
\item
If $\mathfrak{p}$ is a prime hyperideal of $R$ such that $S \cap \mathfrak{p} =\emptyset$, then $S^{-1}\mathfrak{p}$ is a prime hyperideal of $S^{-1}R$.
\item
If $S=R \backslash \mathfrak{p}$ for some prime hyperideal $\mathfrak{p}$ of $R$, then $S^{-1}R$ has the unique maximal hyperideal given by $S^{-1}\mathfrak{p}$.
\end{enumerate}
\end{pro}

\noindent The following theorem of Connes and Consani provides a useful way to construct hyperrings from classical commutative algebras.

\begin{mythm}(cf. \cite[Proposition $2.6$]{con3})\label{connect}
Let $A$ be a commutative ring and \[A^\times:=\{a \in A \mid \exists b \in A\textrm{ such that } ab=1\}.\]
Let $G$ be a (multiplicative) subgroup of $A^\times$ and let $A/G=\{aG \mid a \in A\}$ be the set of cosets.  
\begin{enumerate}
\item
Multiplication:
\[
m:A/G \times A/G \longrightarrow A/G, \quad (xG,yG) \mapsto xyG \quad \forall x, y \in A.\] 
\item
Addition:
\[
+:A/G \times A/G \longrightarrow \mathcal{P}(A/G)^*,\]
such that $(xG,yG)\mapsto \{qG \mid q=xt+ys \textrm{ for some }t,s \in G\} \quad \forall x, y \in A.$ 
\end{enumerate}
Then $A/G$ is a hyperring which is called a quotient hyperring.
\end{mythm}
One can easily observe that for a field $k$ with $|k|\geq 3$, the Krasner's hyperfield $\mathbf{K}$ is isomorphic to the quotient hyperring $k/k^\times$.

\section{Algebraic aspects of Hyperrings}
Hypergroups have been studied by many mathematicians in various fields; harmonic analysis (cf. \cite{liltvinov}), simple groups (cf. \cite{campaigne1940partition}), fuzzy logic (cf. \cite{corsini2003applications}), and association schemes together with Tits' buildings (cf. \cite{zieschangmax}) to name a few. On the other hand, hyperrings have not brought much attention after Krasner's work until recently. In this section, we develop the basic algebraic theory of hyperrings which will be used to develop the geometric theory in the next section.
\subsection{Quotients of hyperrings}
\label{Quotient of hyperrings}
In algebra, the construction of a quotient object is usually essential to develop an algebraic theory. A particular case of quotient construction for hyperrings has been studied by means of normal hyperideals (cf. \cite{Dav2},\cite{Dav1}). Recall that a hyperideal $I$ of a hyperring $R$ is normal if
\[x+I-x \subseteq I \quad \forall x \in R.\]
\begin{rmk}
In \cite{Dav1}, B.Davvaz and A.Salasi introduced the notion of a normal hyperideal $I$ of a hyperring $R$ so that the following relation
\begin{equation}\label{davvreal}
x \equiv y \Longleftrightarrow (x-y)\cap I \neq \emptyset
\end{equation}
becomes an equivalence relation. One may observe that when $R$ is a commutative ring, any ideal of $R$ is normal. In other words, in the classical case, the normal condition is redundant. 
\end{rmk}
\noindent The definition of a normal hyperideal looks too restrictive for applications. For example, suppose that $R$ is a hyperring extension of the Krasner hyperfield $\mathbf{K}$. Then, for any $x \in R$, we have $x+x=\{0,x\}$, therefore $x=-x$. It follows that the only (non-zero) normal hyperideal of $R$ is $R$ itself. In this subsection, we prove that the relation \eqref{davvreal} is, in fact, an equivalence relation without appealing to the normal condition on a hyperideal $I$. Furthermore, we show that one can canonically construct a quotient hyperring $R/I$ for any hyperideal $I$ of a hyperring $R$.\\ 

Let $R$ be a hyperring and $I$ a hyperideal of $R$. We introduce the following relation on $R$ (cf. \cite{Dav1}) 
\begin{equation}\label{relationdef}
 x \sim y \Longleftrightarrow x+I=y+I,
\end{equation}
where $x+I:=\bigcup_{a \in I}(x+a)$ and the equality on the right side of \eqref{relationdef} is meant as an equality of sets. Clearly, the relation \eqref{relationdef} is an equivalence relation.
\begin{rmk}
When $R$ is a commutative ring, \eqref{relationdef} is the classical equivalence relation obtained from an ideal I: $x\sim y \Longleftrightarrow x-y \in I$.
\end{rmk}
\noindent The following lemma provides an equivalent description of \eqref{relationdef}.
\begin{lem}\label{quotientequiclasslem}
Let $R$ be a hyperring and $I$ be a hyperideal of $R$. Let $\sim$ be the relation on $R$ as in \eqref{relationdef}. Then
\begin{equation}\label{relationequiconditon}
x \sim y \Longleftrightarrow (x-y) \cap I \neq \emptyset, \quad \forall x,y \in R.
\end{equation}
\end{lem}
\begin{proof}
Notice that $(x-y)\cap I \neq \emptyset \Longleftrightarrow (y-x)\cap I \neq \emptyset$. Suppose that $x \sim y$. Then by definition we have $x+I=y+I$. By choosing $0 \in I$, it follows that $x+0=x \in y+I$. Thus, $x \in y +a$ for some $a \in I$.
By the reversibility property of $R$, we know that $x \in y+a$ is equivalent to $a \in x-y$.
Thus we derive that $a \in (x-y)\cap I$, hence $(x-y)\cap I \neq \emptyset$.\\ 
Conversely, suppose that $(x-y)\cap I \neq \emptyset$. We need to show that $x+I=y+I$. Since the argument is symmetric, it is enough to show that $x+I \subseteq y+I$. For any $t \in x+I$, there exists $\alpha \in I$ such that $t \in x+ \alpha$. Since $(x-y) \cap I \neq \emptyset$, it follows that there exists $\beta \in (x-y)\cap I$. From the reversibility, this implies that $x \in y +\beta$. Therefore, we have $t \in x+\alpha \subseteq (y+\beta)+\alpha =y+(\alpha + \beta).$ This implies that there exists $\gamma \in (\alpha+\beta)$ such that $t \in y+\gamma$. But since $\alpha, \beta \in I$ we have $\gamma \in I $, thus $t \in y+I$. 
\end{proof}
\noindent Next, we use the equivalence relation \eqref{relationdef} to define quotient hyperrings. We will use the notations $[x]$ and $x+I$ interchangeably for the equivalence class of $x$ under \eqref{relationdef}. We will also use frequently the reversibility property of a hyperring without explicitly mentioning it.

\begin{mydef}\label{quohypdef}
Let $R$ be a hyperring and $I$ be a hyperideal of $R$. We define 
\[ R/I:= \{[x] \mid x \in R\}\]
to be the set of equivalence classes of \eqref{relationdef} on $R$. We impose on $R/I$ two binary operations: an addition:
\begin{equation}\label{quoadd}
[a] \oplus [b]=(a+I) \oplus (b+I):=\{c+I \mid c \in a+b \}
\end{equation}
and a multiplication:
\begin{equation}\label{quomul}
[a] \odot [b] :=a \cdot b + I.
\end{equation}
\end{mydef}

\begin{pro}\label{quotienthyperring}
With the notation as in Definition \ref{quohypdef}, $R/I$ is a hyperring with an addition $\oplus$ and a multiplication $\odot$.
\end{pro}
\begin{proof}
We first prove that operations $\oplus$ and $\odot$ are well defined. For the addition, it is enough to show that $(a+I) \oplus (b+I) = (a' + I) \oplus (b +I)$ for any $[a]=[a']$. In fact, we only have to show one inclusion since the argument is symmetric. Thus, we show that $[a]\oplus[b] \subseteq [a']\oplus[b]$. If $z + I \in (a+I) \oplus (b+I)$, then we may assume $z \in a+b$. We need to show that there exists
$w \in a' +b$ such that $[z]=[w]$. But if $z \in a+b=b+a$ then $a \in z-b$. In particular,
\begin{equation}\label{inclusionquoproof}
(a-a') \subseteq (z-b)-a' = z-(a'+b).
\end{equation}
Since $[a]=[a']$, it follows from Lemma \ref{quotientequiclasslem} that there exists $\delta \in (a-a') \cap I$. It also follows from \eqref{inclusionquoproof} that we have $\delta \in z-w \textrm{ for some } w \in a'+b$ and this implies $(z-w) \cap I \neq \emptyset$. Therefore, we have $[z]=[w]$. For the multiplication, we need to show that $a \cdot b + I = a' \cdot b + I$. Since $(a-a') \cap I \neq \emptyset$, we have $\delta \in (a-a') \cap I \subseteq (a-a')$ which implies that $(a-a')b \cap I \neq \emptyset$. Therefore, $[a\cdot b]=[a'\cdot b]$ from Lemma \ref{quotientequiclasslem}. Hence, $\oplus$ and $\odot$ are well defined.\\
Next, we prove that $(R/I, \oplus)$ is a hypergroup. Clearly $\oplus$ is commutative. We claim that 
\[ X:=([a]\oplus[b])\oplus[c]=\{[d]=d+I\mid d \in a+b+c\}:=Y.\]
If $[w] \in X$, then $[w] \in [r]\oplus[c]$ for some $[r] \in [a]\oplus[b]$. We may assume $w \in r+c$ and $r \in a+b$. Then we have $w \in r+c \subseteq (a+b)+c=a+b+c$. Thus $[w] \in Y$. Conversely, if $[z] \in Y$ then we may assume $z \in a+b+c=(a+b)+c$. This means $z \in t+c$ for some $t \in a+b$. In turn, this implies $[z] \in [t]\oplus[c], [t] \in [a]\oplus[b]$. Hence $[z] \in X$. It follows from the same argument with $[a]\oplus([b]\oplus[c])$ that the operation $\oplus$ is associative. The class $[0]$ is the unique neutral element. In fact, we have $[0]\oplus[x]=\{[d]\mid d \in 0+x=x\}=[x]$. Suppose that we have $[w]$ such that $[w]\oplus[x]=[x]$ for all $[x] \in R/I$. For $x \in I$, we have $x \in w+x=x+w$. Hence $w \in x-x \subseteq I$. But one can see that $[w]=[0]$ for all $w \in I$ from Lemma \ref{quotientequiclasslem}. Therefore, the neutral element is unique. Next, we claim that $[0]\in [x]\oplus [y] \Longleftrightarrow [y]=[-x]$. Since $0 \in (x-x)$ 
we have $[0] \in [x]\oplus[-x]$. Conversely, suppose that $0+I \in (x+I) \oplus (y+I)$ for some $y \in R$. We need
to show that $y+I=-x+I$. Since $0+I \in (x+I) \oplus (y+I)$, there exists $c \in x+y$ such that $c+I=I$. It follows that $c \in I$. Moreover, from $c \in x+y=y-(-x)$, we have that $c \in (y-(-x)) \cap I$. Thus $(y-(-x))\cap I \neq \emptyset$ and $[y]=[-x]$. For the reversibility property, if $[x] \in [y]\oplus[z]$, then we need to show that $[z] \in [x] \oplus [-y]$. But $[x] \in [y]\oplus[z] \Longleftrightarrow (x+I) \in (y+I)\oplus(z+I) \Longleftrightarrow x+I=c+I$ for some $c \in y+z$. From the reversibility property of $R$, $z \in c-y$. Thus $[z] \in [c] \oplus [-y]$. But we have $[x]=[c]$, hence $[z] \in [x]\oplus[-y]$. Finally, we only have to prove that $\oplus,\odot$ are distributive. i.e.
\[([a]\oplus [b])\odot[c]=([a] \odot[c]) \oplus ([b]\odot [c]).\]
But this directly follows from that of $R$. This completes the proof.
\end{proof}
\noindent In the sequel, we consider $R/I$ as a hyperring with the addition $\oplus$ and the multiplication $\odot$.

\begin{pro}\label{quotientstrict}
Let $R$ be a hyperring and $I$ be a hyperideal of $R$. The projection map 
\[\pi:R \longrightarrow R/I, \quad x \mapsto [x]\] 
is a strict, surjective homomorphism of hyperrings with $\Ker \pi =I$.
\end{pro}
\begin{proof}
Clearly, $\pi$ is surjective and $\pi(xy)=\pi(x)\pi(y)$. By the definition of a hyperaddition \eqref{quoadd}, we have $\pi(x+y)=[x+y]\subseteq [x]\oplus [y]$. This shows that $\pi$ is a homomorphism of hyperrings. For the strictness, take $[c] \in [x]\oplus [y]$. Then there exists $z \in x+y$ such that $[z]=[c]$. It follows that $\pi(z)=[z]=[c]$, thus $\pi$ is strict. For the last assertion, suppose that $\pi(x)=[0]$. This implies that $[x]=x+I=0+I=[0]$, hence $x\in I$. Therefore $\Ker\pi=I$.
\end{proof}

\noindent The next proposition shows that a quotient hyperring satisfies the universal property as in the classical case. 

\begin{pro}\label{univproofquotient}
Let $R$ and $H$ be hyperrings and $\varphi: R \longrightarrow H$ be a homomorphism of hyperrings. Suppose that I is a hyperideal of $R$ such that
$I \subseteq \Ker\varphi$. Then there exists a unique hyperring homomorphism $\tilde{\varphi} : R/I \longrightarrow H$ such that $ \varphi= \tilde{\varphi}\circ \pi$, where $\pi :
R \longrightarrow R/I$ is the projection map as in Proposition \ref{quotientstrict}.
\end{pro}
\begin{proof}
Let us define
\[\tilde{\varphi} : R/I \longrightarrow H, \quad \tilde{\varphi}([x])=\varphi(x)\quad \forall [x]\in R/I.\]
We first have to show that $\tilde{\varphi}$ is well defined. Let $[x]=[y]$ for $x,y \in R$. Then we have
$x+I=y+I$, hence $x \in y+c$ for some $c \in I$. Since $c \in I \subseteq \Ker\varphi$, it follows that 
\[\varphi(x) \in \varphi(y+c) \subseteq \varphi(y)+\varphi(c)=\varphi(y)+0=\varphi(y).\]
Therefore, $\varphi(x)=\varphi(y)$ and $\tilde{\varphi}$ is well defined. Furthermore, since $\varphi$ is a hyperring homomorphism, $\tilde{\varphi}$ is also a hyperring homomorphism. By the construction, we have $ \varphi= \tilde{\varphi}\circ \pi$. The uniqueness is clear.
\end{proof}

\subsection{Congruence relations}\label{onetoone}
In this subsection, we define the notion of a congruence relation on a hyperring $R$ and prove that there is a one-to-one correspondence between hyperideals and congruence relations on $R$. Note that in the theory of semirings, this correspondence fails in general as the following example shows. For the basic definitions and properties of semirings we refer the readers to \cite{semibook}.

\begin{myeg}\label{congruenceexample}
Let $M:=\mathbb{Q}_{\geq 0}$ be the semifield of nonnegative rational numbers with the usual addition and the usual multiplication. Since $M$ is a semifield, $\{0\}$ and $M$ are the only ideals of $M$. One can easily see that $\{0\}$ corresponds to the congruence relation: 
\[x \equiv 0 \qquad \forall x \in M\] 
and $M$ corresponds to the congruence relation: 
\[x \equiv y \Longleftrightarrow x=y \qquad \forall x,y \in M.\] 
However there are more congruence relations. For example, one may consider the following relation: $\forall x,y \in M$,
\[x \equiv_2 y \Longleftrightarrow \exists k \in 2\mathbb{Z}+1 \textrm{ such that } k(x-y) \in 2\mathbb{Z}.\]
Clearly, $\equiv_2$ is reflexive and symmetric. Furthermore, suppose that $x\equiv_2 y$ and $y\equiv_2 z$. Then there exist odd integers $k_1$ and $k_2$ such that $k_1(x-y)$, $k_2(y-z) \in 2\mathbb{Z}$. One can easily check that $k_1k_2(x-z) \in 2\mathbb{Z}$. Therefore $\equiv_2$ is an equivalence relation. Next, when $x \equiv_2 y$ and $\alpha \equiv_2 \beta$, $\exists$ odd integers $k$ and $t$ such that $k(x-y),t(\alpha-\beta) \in 2\mathbb{Z}$. It follows that
\[kt((x+\alpha)-(y+\beta))=tk(x-y)+kt(\alpha-\beta) \in 2\mathbb{Z}.\]
Also, one can easily see that $kt(x\alpha-y\beta) \in 2\mathbb{Z}$. Hence, we conclude that 
\[x\equiv_2 y \textrm{ and }\alpha \equiv_2 \beta \Longrightarrow x+\alpha \equiv_2 y+\beta \textrm{ and } x\alpha\equiv_2 y\beta.\]
Therefore $\equiv_2$ is a congruence relation on $M$ which does not have a corresponding ideal of $M$. This example shows that an one-to-one correspondence between ideals and congruence relations fails in this case. In fact, it is well known that if $M$ is a semiring having no nontrivial proper congruence relations then either $M=\mathbb{B}$ (the boolean semifield) or a field (cf.\cite[\S $7$]{semibook}).
\end{myeg}

\noindent We emphasize that in hyperring theory, a sum of two elements is no longer an element in general but a set. Therefore, to define a congruence relation on a hyperring $R$, we need a suitable notion stating when two subsets of $R$ are equivalent. The following definition provides such a notion.

\begin{mydef}
Let $R$ be a hyperring and $\equiv$ be an equivalence relation on $R$. Let $A,B$ be two subsets of $R$. We write $A \equiv B$ when the following condition holds:
\begin{equation}\label{A=B}
\forall a\in A, \forall b \in B\qquad  \exists a' \in A \textrm{ and }\exists b' \in B \textrm{ such that } a\equiv b'\textrm{ and } a'\equiv b.
\end{equation}
\end{mydef}

\begin{mydef}\label{congruencerealdef}
Let $R$ be a hyperring. A congruence relation $\equiv$ on $R$ is an equivalence relation on $R$ satisfying the following property: 
\begin{equation}\label{congeqn}
\forall x_1,x_2,y_1,y_2 \in R,\quad x_1 \equiv y_1,  x_2 \equiv y_2 \quad \Longrightarrow\quad  x_1x_2 \equiv y_1y_2,\quad   x_1+x_2 \equiv y_1+y_2.
\end{equation}
\end{mydef}

\noindent The following proposition shows that when a congruence relation $\equiv$ is defined on $R$, there is a canonical hyperring structure on the set $R/\equiv$ of equivalence classes. We let $[r]$ denote an equivalence class of $r \in R$ under $\equiv$.

\begin{pro}\label{R/= is a hyperring}
The set $(R/\equiv):=\{[r]|\mid r\in R\}$ is a hyperring, where the addition is defined by
\begin{equation}\label{R/= add}
[x]+[y]:=\{[t]\mid t\in x'+y'\quad \forall [x']=[x],[y']=[y]\}\quad \forall x,y,x',y' \in R,
\end{equation}
and the multiplication law is given by
\begin{equation}\label{R/= mult}
[x]\cdot[y]:=[xy] \quad \forall x,y \in R.
\end{equation}
\end{pro}
\begin{proof}
Firstly, we prove that the addition and the multiplication are well defined. One easily sees that \eqref{R/= add} does not depend on representatives since it is already defined by all possible representatives. Also it follows from \eqref{congeqn} that the multiplication is well defined.\\ 
Secondly, we claim that $(R/\equiv,+)$ is a hypergroup. We first show that $+$ is associative by proving the following equality
\[X:=\{[t]\mid t \in x'+y'+z',[x']=[x],[y']=[y],[z']=[z]\}=([x]+[y])+[z]:=Y.\]
Indeed, if $t \in x'+y'+z'$ then $t \in \alpha+z'$ for some $\alpha \in x'+y'$. This implies that $[t] \in [\alpha]+[z]$ and $[\alpha] \in [x]+[y]$, hence $[t] \in Y$. Conversely, if $[t] \in ([x]+[y])+[z]$ then $[t] \in [\alpha]+[z]$ for some $[\alpha] \in [x]+[y]$. From \eqref{R/= add}, we have $t \in \alpha' +z'$ for some $\alpha',z' \in R$ such that $[\alpha']=[\alpha],[z']=[z]$. Also $[\alpha'] \in [x]+[y]$ since $[\alpha]=[\alpha']$. This implies that $\alpha' \in x'+y'$ and $t \in x'+y'+z'$ for some $x',y' \in R$ such that $[x']=[x],[y']=[y]$. The operations are trivially commutative. The class $[0]$ works as the zero element. Indeed, if $[t] \in [x]+[0]$ then $t \in x'+y'$ with $x' \equiv x$ and $y' \equiv 0$. It follows from \eqref{congeqn} that $x'+y' \equiv x$, hence $t \equiv x$. Thus $[x]+[0]=[x]$. An additive inverse of $[x]$ is $[-x]$. Indeed, since $0 \in x-x$ it is clear that $[0] \in [x]+[-x]$. Next, we show that an inverse is unique. If $[0] \in [x]+[y]$ then we have $0 \in x'+y'$ with $x' \equiv x$ and $y' \equiv y$. It follows that $y'=-x'$ and $-x \equiv -x'$, therefore $y \equiv y'=-x'\equiv -x$. Thus an additive inverse uniquely exists. The reversibility property directly follows from that of $R$ and the fact that $[x+y] \subseteq [x]+[y]$. This proves that $(R/\equiv,+)$ is a hypergroup.\\ 
Finally, one can observe that $[1]$  works as the identity element. Therefore, all we have to show is the distributive property:
\[ [z]([x]+[y])=[z][x]+[z][y]=[zx]+[zy],\quad  \forall [x],[y],[z] \in R/\equiv.\]
If $[\alpha] \in [x]+[y]$, then $\alpha \in x'+y'$ with $[x']=[x],[y']=[y]$. This implies $z\alpha \in zx'+zy'$. But since $[zx']=[zx],[zy']=[zy]$, it follows that $[z\alpha] \in [zx]+[zy]$. Conversely if $[t] \in [zx]+[zy]$ then $t \in \alpha +\beta$ with $[\alpha]=[zx]$, $[\beta]=[zy]$. Thus $\alpha+\beta \equiv zx+zy=z(x+y)$, and $t \equiv z\gamma$ for some $\gamma \in x+y$. This completes the proof.   
\end{proof}
\noindent In what follows, for a hyperring $R$ and a congruence relation $\equiv$ on $R$, we always consider $R/\equiv$ as a hyperring with the structure defined in Proposition \ref{R/= is a hyperring}.
\begin{pro}\label{contruencestrict}
Let $R$ be a hyperring and $\equiv$ be a congruence relation on $R$. Then the map 
\[ \pi:R \longrightarrow R/\equiv, \quad r \mapsto [r] \quad \forall r \in R\]
is a strict surjective hyperring homomorphism. 
\end{pro}
\begin{proof}
The map $\pi$ is clearly a surjective hyperring homomorphism. We prove that $\pi$ is also strict by showing that $[x]+[y] \subseteq [x+y]$. If $[t] \in [x]+[y]$ then $t \in x'+y'$ for some $x',y' \in R$ such that $x' \equiv x$ and $y' \equiv y$. It follows from \eqref{congeqn} that $x+y \equiv x'+y'$. From \eqref{A=B}, there exists $\alpha \in x+y$ such that $[\alpha]=[t]$. Therefore, $[t]=[\alpha] \in [x+y]$ and $\pi$ is strict. 
\end{proof}

\begin{pro}\label{idealfromcongruence}
Let $\pi:R \longrightarrow R/\equiv$ be the canonical projection as in Proposition \ref{contruencestrict}. Let $I=\Ker\pi$. Then
\[\varphi:R/I \longrightarrow R/\equiv,\quad  <r> \mapsto [r]\quad \forall r \in R\] is an isomorphism of hyperrings, where $<r>$ is an equivalence class of $r$ in $R/I$ under the equivalence relation \eqref{relationdef} and $[r]$ is an equivalence class of $r$ in $R/\equiv$ under $\equiv$.
\end{pro}
\begin{proof}
This follows from Proposition \ref{contruencestrict} and Proposition $2.11$ of \cite{Dav1} which states that the first isomorphism theorem for hyperrings holds when a given homomorphism is strict.
\end{proof}

\noindent It follows from Proposition \ref{idealfromcongruence} that for a congruence relation $\equiv$ on R, one can find a hyperideal $I$ of $R$ such that $R/I \simeq (R/\equiv)$. Conversely, in the next proposition, we prove that for any hyperideal $I$, one can find a congruence relation $\equiv$ such that $R/I \simeq (R/\equiv)$. 

\begin{rmk}
Note that some of the algebraic properties of a hyperring differ greatly from those of a commutative ring. For example, a hyperring does not satisfy the doubly distributive property (cf. Remark \ref{double}). Thus one should be careful when generalizing classical results of commutative rings to hyperrings.
\end{rmk}

\begin{pro}\label{fromidealtocongruence}
Let $R$ be a hyperring and $I$ be a hyperideal of $R$. Then the relation $\equiv$ such that 
\[x \equiv y \Longleftrightarrow x+I=y+I\]
is a congruence relation and $R/I \simeq (R/\equiv)$.
\end{pro}
\begin{proof}
Clearly $\equiv$ is an equivalence relation. If $x_1 \equiv y_1$ and $x_2 \equiv y_2$, we have 
\begin{equation}\label{congruproof}
x_i+I=y_i+I, \quad i=1,2. 
\end{equation}
Thus we can find $\alpha, \beta \in I$ such that $x_1 \in y_1+\alpha$, $x_2 \in y_2+\beta$. By multiplying these two, one obtains
\[x_1x_2 \in (y_1+\alpha)(y_2+\beta) \subseteq y_1y_2 +y_1\beta+y_2\alpha+\alpha\beta.\]
Therefore, for any $t \in I$, we have $x_1x_2+t \subseteq y_1y_2 +(y_1\beta+y_2\alpha+\alpha\beta +t)$. But since $\alpha,\beta,t \in I$, it follows that $(y_1\beta+y_2\alpha+\alpha\beta +t) \subseteq I$. Hence, 
$x_1x_2 +t \subseteq y_1y_2 +I$ and $x_1x_2 +I \subseteq y_1y_2+I$. Since the argument is symmetric, we have
\[x_1x_2+I=y_1y_2+I \Longleftrightarrow x_1x_2 \equiv y_1y_2.\]
For the other condition of a congruence relation, we
need to show $(x_1+x_2) \equiv (y_1+y_2)$. 
It is enough to show that $\forall$ $t \in x_1+x_2$, there exists $y \in y_1+y_2$ such that $t \equiv y$. We can take $\alpha, \beta \in I$ such that $x_1 \in y_1+\alpha$, $x_2 \in y_2 +\beta$ from \eqref{congruproof}. It follows that
\[t \in (x_1+x_2) \subseteq (y_1+y_2)+(\alpha+\beta).\]
Hence, $t \in y+\gamma$ for some $y \in y_1+y_2$, $\gamma \in \alpha +\beta \subseteq I$. This implies that $t \equiv y$ from \eqref{relationequiconditon} and the reversibility property of $R$. It is clear that in this case the kernel of a canonical projection map $\pi:R \longrightarrow R/\equiv$ is $I$. It follows from the first isomorphism theorem of hyperrings (cf. \cite[Proposition $2.11$]{Dav1}) that $R/I \simeq R/\equiv$ since $\pi$ is strict.
\end{proof}

\begin{rmk}
Let $R$ be a hyperring and $I$ be a hyperideal of $R$. In a quotient hyperring $R/I$, we defined the addition as \[[a]\oplus [b]=\{[c]|c\in a+b\}\] 
and we proved that $x\sim y \Longleftrightarrow x+I=y+I$ is a congruence relation. In this case, we defined the addition as 
\[[a]+[b]=\{[c]\mid c\in a'+b'\quad \forall [a']=[a],[b']=[b]\}.\] At first glance, $[a]\oplus [b]$ and $[a]+[b]$ seem different, but in fact they are the same sets. Clearly $[a]\oplus [b] \subset [a]+[b]$. Conversely, assume that $t' \in a'+b'$ for some $[a']=[a]$, $[b']=[b]$. Since $a'+I=a+I$, $b'+I=b+I$, we can find $\alpha,\beta \in I$ such that $a' \in a+\alpha,b' \in b+\beta$. This implies that $t' \in a'+b'\subseteq (a+b)+(\alpha+\beta)$. But since $(\alpha+\beta)\subseteq I$, it follows that $t' \in t+\gamma$ for some $t \in (a+b), \gamma \in I$. By the reversibility property of $R$, $\gamma \in t'-t$. In other words, $(t-t')\cap I \neq \emptyset$, hence $[t]=[t']$. This shows that $[a]+[b] \subseteq [a]\oplus [b]$.
\end{rmk}
\subsection{Analogues of classical lemmas}\label{lem}
In this subsection, we reformulate several basic results in commutative algebra in terms of hyperrings. Throughout this subsection, we denote by $R$ a hyperring and by $V(I)$ the set of of prime hyperideals of $R$ containing a hyperideal $I$. We also denote by $Nil(R)$ the intersection of all prime hyperideals of $R$.

\begin{lem}\label{radical}
Let $I \subseteq R$ be a hyperideal. Then the following set:
\[ \sqrt{I}:=\{ r\in R \mid \exists n \in \mathbb{N} \textrm{ such that  } r^n \in I\}\]
is a hyperideal.
\end{lem}
\begin{proof}
Trivially we have $0 \in \sqrt{I}$. Suppose that $a \in \sqrt{I}$, then $a^n \in I$ for some $ n \in \mathbb{N}$. Since $I$ is a hyperideal, for $r \in R$,  we have $r^na^n=(ra)^n \in I$. It follows that $ra \in \sqrt{I}$. Clearly, $(-a)^n$ is either $a^n$ or $-a^n$. Since both $a^n$ and $-a^n$ are in $I$, it follows that $-a \in \sqrt{I}$. Finally, suppose that $a,b \in \sqrt{I}$ and $a^n,b^m \in I$. Then, for $l\geq (n+m)$, we have $(a+b)^l \subseteq \sum {l \choose k}a^kb^{l-k} \subseteq I$. This implies that $(a+b) \subseteq \sqrt{I}$; therefore, $\sqrt{I}$ is a hyperideal.
\end{proof}

\begin{rmk}\label{double}
In general, a hyperring does not satisfy the doubly distributive property (cf. \cite[pp $13-14$]{viro}). In other words, the following identity:
\[ (a+b)(c+d)=ac+ad+bc+bd\]
is in general not fulfilled. Instead, the following identity:
\[ (a+b)(c+d) \subseteq ac+ad+bc+bd\] 
holds.
\end{rmk}

\begin{lem}\label{nilradical}
Let $R$ be a hyperring and $I$ a hyperideal of $R$. Then
\[ \sqrt{I}=\bigcap_{\mathfrak{p}\in V(I)} \mathfrak{p}. \]
\end{lem}

\begin{proof}
Suppose that $a \in \sqrt{I}$, then $a^n \in I \subseteq \mathfrak{p}$ for all $\mathfrak{p} \in V(I)$. Since $\mathfrak{p}$ is a prime hyperideal, it follows that $a \in \mathfrak{p}$; hence, $\sqrt{I} \subseteq \mathfrak{p}$ for all $\mathfrak{p} \in V(I)$.\\ 
Conversely, suppose that $f \in \bigcap_{\mathfrak{p}\in V(I)} \mathfrak{p}$ and $f \not\in \sqrt{I}$. This implies that 
\[S:=\{1,f,f^2,....\} \cap I =\emptyset.\] 
Let $\Sigma$ be the set of hyperideals $J$ of $R$ such that $S \cap J = \emptyset$ and $I \subseteq J$. Then $\Sigma \neq \emptyset$ since we have $\sqrt{I} \in \Sigma$. By Zorn's lemma (ordered by inclusion), $\Sigma$ has a maximal element $\mathfrak{q}$. Then $\mathfrak{q}$ is a prime hyperideal. Indeed, by definition, $\mathfrak{q}$ is a hyperideal. Therefore, all we have to prove is that $\mathfrak{q}$ is prime. One can easily check, for $x \in R$, the following set:
\[ \mathfrak{q}+xR:=\bigcup \{a+b \mid a \in \mathfrak{q}, b \in xR \}\]
is a hyperideal. If $x, y \not\in \mathfrak{q}$ then $\mathfrak{q}$ is properly contained in $\mathfrak{q}+xR$ and $\mathfrak{q}+yR$. Thus, $\mathfrak{q}+xR$, $\mathfrak{q}+yR$ $\not\in \Sigma$ from the maximality of $\mathfrak{q}$ in $\Sigma$. It follows that $f^n \in \mathfrak{q}+xR$ and $f^m \in \mathfrak{q}+yR$ for some $n,m \in \mathbb{N}$. In other words, $f^n \in a_1+xr_1$, $f^m \in a_2 +yr_2$ for some $a_1,a_2 \in \mathfrak{q}$ and $r_1,r_2 \in R$. Therefore, we have
\[f^{n+m} \in (a_1+xr_1)(a_2+yr_2)\subseteq a_1a_2+a_1yr_2+a_2xr_1+xyr_1r_2 \subseteq \mathfrak{q}+xyR.\]
This implies that $xy \not \in \mathfrak{q}$ because if $xy \in \mathfrak{q}$ then $f^{n+m} \in \mathfrak{q}$, and we assumed that $f^l \not \in \mathfrak{q}$ for all $l \in \mathbb{N}$. It follows that $\mathfrak{q}$ is a prime hyperideal containing $I$ such that $S \cap \mathfrak{q} = \emptyset$. However, this is impossible since we took $f \in \bigcap_{\mathfrak{p} \in V(I)}\mathfrak{p}$. This completes the proof.
\end{proof}

\noindent For a family $\{X_\alpha\}_{\alpha \in J}$ of subsets $X_{\alpha}\subseteq R$, we denote by $<X_\alpha>_{\alpha \in J}$ the smallest hyperideal of $R$ containing $(\bigcup_{\alpha \in J}X_{\alpha})$. Note that $<X_\alpha>_{\alpha \in J}$ always exists since an intersection of hyperideals is a hyperideal as in the classical case. We call $<X_\alpha>_{\alpha \in J}$ the hyperideal generated by $\{X_\alpha\}_{\alpha \in J}$.

\begin{lem}\label{generatingideal}
Let $J$ be an index set.
\begin{enumerate}
\item
Let $h \in R$. Then the hyperideal generated by $h$ is 
\begin{equation}
hR:=\{hr \mid r \in R\}.
\end{equation}
\item
Suppose that $I_i$ is the principal hyperideal generated by an element $h_i \in R$ for each $i \in J$. Then
\begin{equation}\label{hypegen}
<I_i>_{i \in J}=\{r \in R \mid r \in \sum_{i \in X} b_ih_i, b_i \in R, X\subseteq J, |X|<\infty \}. 
\end{equation}
\item
Let $\{I_i\}_{i \in J}$ be a family of hyperideals $I_j \subseteq R$. Then 
\begin{equation}\label{hypergen2} 
<I_i>_{i\in J}=\{r \in R \mid r \in \sum_{i\in X} b_ih_i, b_i \in R,h_i \in I_i,X\subseteq J, |X|<\infty\}.  
\end{equation}
\end{enumerate}
\end{lem}

\begin{proof}
The proof is similar to the classical case.
\end{proof}

\noindent Let $R^{\times}:=\{r \in  R \mid rr'=1 \textrm{ for some } r' \in R\}$ and $J(R)$ be the intersection of all maximal hyperideals of $R$. The following lemma has been proven in \cite{Dav1}.

\begin{lem}\label{jacobsonradical}(\cite[Proposition $2.12$, $2.13$, $2.14$]{Dav1})
\begin{enumerate}
\item
$x \in J(R) \Longleftrightarrow(1-xy) \subseteq R^{\times} \quad \forall y \in R.$
\item
For any hyperideal $I\subsetneq R$, we have $V(I) \neq \emptyset$.
\end{enumerate}
\end{lem}

\noindent One imposes the Zariski topology on the set $\Spec R$ of prime hyperideals of $R$ as in the classical case (cf. Definition \ref{zariskitop}, Proposition \ref{specRcomputation}). In what follows, we consider $X=\Spec R$ as a topological space equipped with the Zariski topology. Then, as in classical algebraic geometry, we have the following.

\begin{pro}\label{discon}
$X=\Spec R$ is a disconnected topological space if and only if $R$ has a (multiplicative) idempotent element different from $0,1$.
\end{pro}
\begin{proof}
Suppose that $e\neq 0,1$ is an idempotent element of $R$. Then we have $e^2=e$, and it follows that $0 \in e(e-1)$. Therefore there is an element $f \in e-1$ such that $ef=0$. Moreover, $f\neq 0$ since $e \neq 1$. Similarly, $f$ can not be $1$ since $ef=0$ and $e \neq 0$. Together with Lemma \ref{jacobsonradical}, it follows that $V(e)$ and $V(f)$ are nonempty subsets of $X$. Hence we have $X=\Spec R=V(e) \cup V(f)$ since $ef=0$. \\
Now we claim that $V(e) \cap V(f) =\emptyset$. Indeed, if $\mathfrak{p} \in V(e) \cap V(f)$, then $e,f \in \mathfrak{p}$. This implies that $-e,-f \in \mathfrak{p}$ and $(f-e) \subseteq \mathfrak{p}$. However, we have \[
f \in e-1=-1+e \Longleftrightarrow -1 \in f-e.
\] 
Thus we should have $1 \in \mathfrak{p}$ and it is impossible. It follows that $\{V(e)^c,V(f)^c\}$ becomes the disjoint open cover of $X$ and hence $X$ is disconnected.\\
Conversely, suppose that $X=\Spec R=U_1 \cup U_2$, where $U_1$ and $U_2$ are disjoint open subsets of $X$. In particular, $U_1$ and $U_2$ are also closed. We may assume the following (cf. Proposition \ref{specRcomputation}): for some hyperideals $I$ and $J$,
\[ X=\Spec R=V(I) \cup V(J)=V(IJ), \quad V(I) \cap V(J)= V(<I,J>)=\emptyset.\]
Recall that $Nil(R):=\bigcap_{\mathfrak{p} \in X}\mathfrak{p}$. It follows from Lemma \ref{radical} and \ref{nilradical} that $Nil(R)$ is the set of all nilpotent elements of $R$. Since $V(IJ)=X=V(Nil(R))$, we have $\sqrt{IJ}=Nil(R)$ from Lemma \ref{nilradical}. Moreover, the fact $V(<I,J>)=\emptyset$ implies that $\sqrt{<I,J>}$ contains $1$. Otherwise, $\sqrt{<I,J>}$ does not contain any unit element, and $V(<I,J>)\neq \emptyset$ from Lemma \ref{jacobsonradical}. It follows that $1 \in \sqrt{<I,J>}$, hence $1 \in <I,J>$. From Lemma \ref{generatingideal}, there exist $a \in I$ and $b \in J$ such that $1 \in a+b$.\\ 
We claim that $a, b \not \in Nil(R)$. Indeed, suppose that $a \in Nil(R)$. Then $a \in \mathfrak{p}$ for any $\mathfrak{p} \in X$. This implies that for $\mathfrak{p} \in V(J)$, $\mathfrak{p}$ contains both $a$ and $b$ and hence $1 \in (a+b)\subseteq \mathfrak{p}$.\\
Next, since 
\[V(a) \supseteq V(I) \Longleftrightarrow D(a) \subseteq (V(I))^c=V(J),\] 
we have 
\[D(a) \subseteq V(J),\quad D(b) \subseteq V(I).\] 
This implies 
\[(V(a) \cup V(b))^c=D(a) \cap D(b) \subseteq V(I) \cap V(J) =\emptyset,\] 
thus $V(a) \cup V(b) =X$.
Suppose that $A=<a>$ and $B=<b>$. Then $ab \in A \cap B$, and it follows that $V(a) \cup V(b) = V(A) \cup V(B)=V(AB)$. Thus we have $AB \subseteq  \sqrt{AB} =Nil(R)$. Therefore, $ab \in Nil(R)$, in turn, $(ab)^n=a^nb^n=0$ for some $n \in \mathbb{N}$. However, $a^n, b^n \neq 0$ since $a, b \not \in Nil(R)$. We observe the following:
\begin{equation}\label{proofusage}
1 \in a+b \Longrightarrow 1 \in (a+b)^n \subseteq \sum d_k a^k b^{n-k} \Longrightarrow 1 \in (a^n+b^n)+abf, \textrm{ for some }f \in R.
\end{equation}
Since $ab \in Nil(R)$, clearly $abf \in Nil(R) \subseteq J(R)$. It follows from \eqref{proofusage} that
\[ 1 \in \alpha + abf, \quad \textrm{ for some } \alpha \in a^n+b^n.\]
This implies that $\alpha \in 1-abf$. But since $abf \in J(R)$, from Lemma \ref{jacobsonradical}, $\alpha$ is a unit. Let $\beta=\alpha^{-1}$. Then
\[ \alpha \in a^n+b^n \Longrightarrow \alpha\beta=1 \in (a^n+b^n)\beta=a^n\beta+b^n\beta \Longrightarrow b^n\beta \in 1-a^n\beta.\]
One observes that $a^n\beta, b^n\beta \neq 0$ since $a^n,b^n \neq 0$ and $\beta$ is a unit. Furthermore, $a^n\beta, b^n\beta \neq 1$. Since $a^n\beta=1 \Longleftrightarrow a^n=\beta^{-1}=\alpha$, it would imply that $a^n=\alpha \in I$. Therefore $V(I)=\emptyset$. But we assumed that $V(I) \neq \emptyset$.\\ 
Finally let us define an element $e=a^n\beta$. Then we know, from the above, $e \neq 0,1$. Furthermore, we have
\[e^2-e=e(e-1)=a^n\beta(a^n\beta-1).\]
Since we have  $b^n\beta \in 1-a^n\beta$ and $b^n\beta a^n\beta=a^nb^n \beta^2=0$, it follows that 
\[0 \in e(e-1)=e^2-e.\] 
Hence, from the uniqueness of an inverse, we have $e^2=e$ and $e\neq 0, 1$.
\end{proof}

\begin{pro}\label{irreducible}
$X=\Spec R$ is irreducible if and only if $Nil(R)$ is a prime hyperideal.
\end{pro}
\begin{proof}
The proof is similar to the classical case.
\end{proof}

\section{Geometric aspects of Hyperrings}
In this section, we investigate geometric aspects of hyperrings. Our goal is to reformulate tropical varieties in the framework of hyperfields and to realize a hyperring $R$ (without zero-divisors) as the hyperring of global regular functions on the topological space $X=\Spec R$. 
\begin{rmk}
It is noteworthy that finite hyperfield extensions of the Krasner hyperfield $\mathbf{K}$ already have very interesting geometric aspect, namely a correspondence to incidence geometry. More precisely, in \cite{con3}, Connes and Consani proved that for a finite hyperfield extension $R$ of the Krasner hyperfield $\mathbf{K}$, one can associate an incidence geometry (with some extra conditions) and vice versa. Furthermore, under their construction, the open problem of completely classifying finite hyperfield extensions of $\mathbf{K}$ is very closely related with a long-standing conjecture on the existence of a non-Desarguesian finite projective plane with a simply transitive group of collineations. For more details, see \cite{con3}, \cite{thas2014hyperfield}, \cite{thas2008finite}. Also, for more connections between incidence geometry and hypergroups, we refer the reader to \cite{corsini2003applications}.
\end{rmk} 
\subsection{A tropical variety as an algebraic set over a hyperfield}
In this subsection, we recast a tropical variety as the `positive part' of an algebraic set over a hyperfield. For more details about tropical geometry we refer the reader to \cite{bernd}. Note that our notion of a tropical variety in this subsection (cf. Equation \eqref{tropicalhypersurface}, Remarks \ref{tropdifferconventional}) a bit differs from the one given in \cite{bernd}. However, such choice makes no difference in further study. Also, strictly speaking, a tropical variety is the support of a polyhedral complex with a balancing condition, but we only consider a tropical variety as a set without a balancing condition.\\
The main motivation for the study proposed in this section comes from the following observation: the definition of a tropical variety does not seem natural in the sense that it is not defined as the set of solutions of polynomial equations, but as the set of points where a maximum (or a minimum depending on the conventions) is attained at least twice. Recently, there have been several attempts to build an algebraic foundation of tropical geometry: for instance, \cite{noah}, \cite{iza}, \cite{oliver1}, \cite{Payne}, \cite{viro}. We take up a hyperstructural approach and, in Proposition \ref{tropicalvareityispositivepart}, we show that there exists a more natural description of a tropical variety by implementing a symmetrization procedure and an associated algebraic set over a hyperfield.\\ 

In this subsection, we use the multi-index notation: for $I=(i_1,...,i_n) \in \mathbb{N}^n$, $X^I:=x_1^{i_1}\cdot\cdot\cdot x_n^{i_n}$. Let us first define the notion of a polynomial equation with coefficients in a hyperring $R$.
\begin{mydef}\label{hyperpoly}
Let $R$ be a hyperring. 
\begin{enumerate}
\item
By a monomial $f$ with $n$ variables over $R$ we mean a single term:
\begin{equation}\label{monomial}
f:=a_IX^I, \qquad a_I \in R, \quad I \in \mathbb{N}^n.
\end{equation}
\item
By a polynomial $f$ with $n$ variables over $R$ we mean a finite formal sum:
\begin{equation}\label{hyppolydef}
f:=\sum_{I \in \mathbb{N}^n} a_I X^I, \qquad a_I \in R,\quad  a_I=0 \textrm{ for all but finitely many $I$}
\end{equation}
such that there is no repetition of monomials with the same multi-index $I$. We denote by $R[x_1,...,x_n]$ the set of polynomials with $n$ variables over $R$.
\end{enumerate}
\end{mydef} 
\noindent One can be easily misled in the above definition. For example, $(x-x)$ is not a polynomial over the hyperfield of signs $\mathbf{S}$ (Example \ref{sign}) since the term $x$ is repeated. The reason why we do not want $(x-x)$ to be a polynomial is that whenever a repetition of a monomial occurs, an ambiguity follows. For instance, we may have $(x-x)=(1-1)x=\{-x,0,x\}$. In other words, $(x-x)$ does not represent a single element. Furthermore, one can not perform the basic arithmetic in general. For example, $(x^2-1)$ differs from $(x+1)(x-1)$ as an element of $\mathbf{S}[x]$ (cf. Example \ref{multivaluedmult}). Note that $(x+1)(x-1)$ is not even a polynomial over $\mathbf{S}$ since it is not of the form \eqref{hyppolydef} Indeed it represents several different polynomials as you can see in Example \ref{multivaluedmult}. Therefore, for $f,g \in R[x_1,...,x_n]$, we say $f=g$ only if they are entirely identical.\\

\noindent We directly generalize the classical addition and multiplication of polynomial equations to $R[x_1,...,x_n]$. For example, for $f=\sum_{i=0}^na_ix^i$, $g=\sum_{j=0}^mb_jx^j \in R[x]$, the addition and the multiplication of $f$ and $g$ are given by
\begin{equation}\label{additionandmultiplicationofpoly}
f+g:=\sum_{i=0}^n(a_i+b_i)x^i+\sum_{i=n+1}^mb_ix^i,\qquad fg:=\sum_{i=0}^{n+m}(\sum_{r+l=i}a_rb_l)x^i. \quad (n\leq m)
\end{equation}
Notice that $(a_i+b_i)$ and $\sum_{r+l=i}a_rb_l$ are not elements but subsets of $R$ in general. Therefore, the addition and the multiplication defined in this way are in general multi-valued as the following example shows.

\begin{myeg} \label{multivaluedmult}
Let $R=\mathbf{S}$, the hyperfield of signs. For $x-1, x+1 \in \mathbf{S}[x]$, we have
\[(x+1)(x-1)=x^2+(1-1)x-1=\{x^2-1,x^2+x-1,x^2-x-1\}.\]
\[(x+1)+(x-1)=(x+x)+(1-1)=\{x,x+1,x-1\}.\]
\end{myeg}

\begin{rmk}\label{plyisset}
We emphasize that $R[x_1,...,x_n]$ is only a set with two multi-valued binary operations. However it appears, in some circumstance, that $R[x_1,...,x_n]$ behaves like a hyperring as Example \ref{multiringpoly} shows.
\end{rmk}

\noindent Throughout this section, we will simply write a polynomial over $R$ instead of a polynomial with $n$ variables when there is no possible confusion.

\begin{mydef}\label{evaluation}
Let $R$ be a hyperring and $R[x_1,...,x_n]$ be the set of polynomials over $R$. Let $L$ be a hyperring extension of $R$. By an evaluation $f(\alpha)$ of $f=\sum_I a_I X^I \in R[x_1,...,x_n]$ at $\alpha=(\alpha_1,...,\alpha_n) \in L^n$ we mean the following set:
\begin{equation}\label{evaluationequation}
f(\alpha):=\sum_I a_I\alpha^I \subseteq L.
\end{equation}
\end{mydef}

\begin{myeg}
Let $R=L=\mathbf{S}$, the hyperfield of signs. Suppose that  $f=x^2-x \in R[x]$. Then: $f(1)=\mathbf{S}, \quad f(0)=\{0\}, \quad f(-1)=\{1\}.$\\
Let $L=\mathcal{T}\mathbb{R}$, Viro's hyperfield (Example \ref{viro}). Then: $f(1)=[-1,1],\quad f(0)=\{0\}, \quad f(-1)=\{1\}.$
\end{myeg}

\noindent An intuitive definition of a set of solutions of $f \in R[x_1,...,x_n]$ over a hyperring extension $L$ of $R$ would be the following set:
\begin{equation}\label{wrongdef}
\{\alpha=(\alpha_1,...,\alpha_n) \in L^n \mid 0 \in f(\alpha)\}.
\end{equation}
However, \eqref{wrongdef} may depend on the way one writes $f$ (cf. \cite[\S $5.2$]{viro2}). Moreover, for two different elements $f,g \in R[x_1,...,x_n]$, we may have $f(\alpha)=g(\alpha)$ $\forall \alpha \in L^n$. For example, suppose that $f=x^2-1$, $g=x^4-1 \in \mathbf{S}[x]$. Then $f(a)=g(a)$ $\forall a \in \mathcal{T}\mathbb{R}$, but $f\neq g$ as elements of $\mathbf{S}[x]$. To resolve these issues, we introduce the following relation on $R[x_1,...,x_n]$.

\begin{mydef}
Let $R$ be a hyperring and $R[x_1,...,x_n]$ be the set of polynomials over $R$. Let $L$ be a hyperring extension of $R$. For $f,g \in R[x_1,...,x_n]$, we define 
\begin{equation}\label{equivofpoly}
f \equiv_L g \Longleftrightarrow f(\alpha)=g(\alpha) \textrm{ (as sets) } \quad \forall \alpha \in L^n.
\end{equation}
\end{mydef}

\begin{rmk}
The relation \eqref{equivofpoly} depends on a hyperring extension $L$ of $R$. However, we note that if $H$ is a hyperring extension of $L$ then $f\equiv_H g \Longrightarrow f\equiv_L g.$
\end{rmk}

\noindent The following statement is clear in view of the above definition.

\begin{pro}
Let $R$ be a hyperring and $L$ be a hyperring extension of $R$. Then the relation \eqref{equivofpoly} on $R[x_1,...,x_n]$ is an equivalence relation.
\end{pro}
\begin{proof}
This is straightforward since \eqref{equivofpoly} is defined in terms of an equality of sets.
\end{proof}

\noindent Under the equivalence relation \eqref{equivofpoly}, we can consider each polynomial $f \in R[x_1,...,x_n]$ as the following function: 
\begin{equation}\label{polyfunction}
f:L^n \longrightarrow P^*(L), \qquad \alpha=(\alpha_1,...,\alpha_n) \mapsto f(\alpha),
\end{equation}
where $P^*(L)$ is the set of nonempty subsets of $L$. 

\begin{mydef}
Let $R$ be a hyperring and $R[x_1,...,x_n]$ be the set of polynomials over $R$. Let $L$ be a hyperring extension of $R$. By a solution of $f$ over $L$ we mean an element $a=(a_1,...,a_n) \in L^n$ such that $0 \in f(a)$ where we consider $f$ as in \eqref{polyfunction} under $\equiv_L$. We denote by $V_L(f)$ the set of solutions of $f$ over $L$.
\end{mydef}
\begin{rmk}
Suppose that $H$ is a hyperring extension of $L$. It clearly follows from the definition that $V_L(f) \subseteq V_H(f)$.
\end{rmk}

\noindent Let $R[x_1,...,x_n]/\equiv_L$ be the set of equivalence classes of $R[x_1,...,x_n]$ under $\equiv_L$. When $L$ is doubly distributive (hence, so is $R$), the multiplication on $R[x_1,...,x_n]/\equiv_L$ induced from the multi-valued multiplication on $R[x_1,...,x_n]$ is well defined and single-valued (cf. Remark \ref{double} for the definition of doubly distributive property). In fact, suppose that $f, g \in R[x_1,...,x_n]$. If $h \in f\cdot g$ then it follows from the doubly distributive property of $L$ that $h(\alpha)=f(\alpha)\cdot g(\alpha)$ for all $\alpha \in L^n$ (cf. \cite[the remark after Theorem $4.B.$]{viro}). Therefore, under $\equiv_L$, the set $f\cdot g$ becomes a single equivalence class. This is one of the advantages of working with $R[x_1,...,x_n]/\equiv_L$ rather than working directly with $R[x_1,...,x_n]$. 

\begin{myeg}
Let $R=\mathbf{S}$, the hyperfield of signs and $L=\mathcal{T}\mathbb{R}$, Viro's hyperfield. Then $L$ satisfies the doubly distributive property (cf. \cite[Theorem $7.B.$]{viro}). In Example \ref{multivaluedmult}, we computed $(x+1)(x-1)=\{x^2-1,x^2+x-1,x^2-x-1\}$ in $\mathbf{S}[x]$. One can easily see that
\[\forall a \in \mathcal{T}\mathbb{R}, \quad 
(a^2-1)=(a^2+a-1)=(a^2-a-1)=\left\{ \begin{array}{lll}
a^2 \quad \textrm{ if $|a|>1$}\\
\lbrack-1,1\rbrack \quad \textrm{ if $|a|=1$}\\
-1 \quad \textrm{ if $|a|<1$}
\end{array} \right.\]
Therefore $(x^2-1)\equiv_{\mathcal{T}\mathbb{R}} (x^2+x-1) \equiv_{\mathcal{T}\mathbb{R}} (x^2-x-1)$, and $x=1,-1$ are the only solutions of the equivalence class of $(x^2-1)$ under $\equiv_{\mathcal{T}\mathbb{R}}$.
\end{myeg}

\noindent The following example shows that in general one can not expect $R[x_1,...,x_n]/\equiv_L$ to be a hyperring even when $L$ satisfies the doubly distributive property.

\begin{myeg}\label{multiringpoly}
Let $R=L=\mathbf{K}$, the Krasner's hyperfield. Let $[f]$ be the equivalence class of $f \in \mathbf{K}[x]$ under $\equiv_\mathbf{K}$. Then any two non-constant polynomials over $\mathbf{K}$ with the same constant term are equivalent under $\equiv_\mathbf{K}$. It follows that
\[(\mathbf{K}[x]/\equiv_\mathbf{K})=\{[0],[1],[1+x],[x]\}.\]
For the notational convenience, let $0:=[0],1:=[1],a:=[1+x]$, and $b:=[x]$. Then we have the following tables:
\[
\begin{tabular}{ | c | c | c | c | c |}
    \hline
    $+$ & $0$ & $1$ & a & b \\ \hline
    $0$ & $0$ & $1$ &  a & b\\ \hline
    $1$ & $1$ & $\{0,1\}$ & \{a,b\} & a\\ \hline
    a & a & $\{a,b\}$ & $\{0,1,a,b\}$ & $\{1,a\}$\\ \hline
    b & b & a & $\{1,a\}$ & $\{0,b\}$\\
    \hline
    \end{tabular} \qquad 
    \begin{tabular}{ | c | c | c | c | c |}
        \hline
        $\cdot$ & $0$ & $1$ & a & b \\ \hline
        $0$ & $0$ & $0$ &  $0$ & $0$\\ \hline
        $1$ & $0$ & $1$ & $a$ & $b$\\ \hline
        a & $0$ & $a$ &  $a$ & $b$\\ \hline
        b & $0$ & $b$ &  $b$ & $b$\\
        \hline
        \end{tabular}
\]
One can check by using the above tables that $(\mathbf{K}[x]/\equiv_{\mathbf{K}},+)$ is a hypergroup, but it fails to satisfy the distributive law. For example, we have
\[a(1+b)=\{a\} \subseteq a+ab=\{1,a\}.\] 
However, $(\mathbf{K}[x]/\equiv_{\mathbf{K}},+,\cdot)$ still satisfies the weak version of the distributive law: $a(b+c) \subseteq ab+ac$. It follows that $(\mathbf{K}[x]/\equiv_{\mathbf{K}},+,\cdot)$ is not a hyperring but a multiring introduced in \cite{mars1} which assumes the same axioms with a hyperring but the distributive law is weaker as above.\\ 
We recall that if $A$ and $B$ are multirings, one defines a homomorphism of multirings as a map $\varphi:A\longrightarrow B$ such that 
\[\varphi(a+b)\subseteq \varphi(a)+\varphi(b),\quad \varphi(ab)=\varphi(a)\varphi(b), \quad \varphi(1_A)=1_B, \quad \varphi(0_A)=0_B \qquad \forall a,b \in A.\]
Consider the set $\mathbb{A}^1_{\mathbf{K}}(\mathbf{K}):=\Hom_{multi}(\mathbf{K}[x]/\equiv_\mathbf{K},\mathbf{K})$ of multiring homomorphisms from $\mathbf{K}[x]/\equiv_\mathbf{K}$ to $\mathbf{K}$. If $\varphi \in \mathbb{A}^1_{\mathbf{K}}(\mathbf{K})$ then one has $\varphi(a)=1$ since $a+a=(\mathbf{K}[x]/\equiv_{\mathbf{K}})$. One can also easily check that $\varphi(b)$ can be any point of $\mathbf{K}$. It follows that $\mathbb{A}^1_{\mathbf{K}}(\mathbf{K})=\{\varphi_0,\varphi_1\}$, where $\varphi_0(b)=0$ and $\varphi_1(b)=1.$ This suggests that one may consider $\mathbf{K}[x]/\equiv_\mathbf{K}$ as the `coordinate ring' of an affine line over $\mathbf{K}$.
\end{myeg}

\noindent In the sequel, we will always consider an element $f$ of $R[x_1,...,x_n]$ under the equivalence relation \eqref{equivofpoly} with a predesignated hyperring extension $L$ of $R$. We will use the symmetrization process introduced in \cite{Henry} and the previous result of the author in \cite{jaiungthesis}. Let us briefly explain the symmetrization process. For more details, see Appendix \ref{app}.\\ 

Let $M$ be a totally ordered idempotent monoid (with a canonical order). We define
\[s(M):=\{(s,1),(s,-1),0=(0,1)=(0,-1) \mid s \in M\backslash \{0\}\}.\]
Also, we denote $(s,1):=s$, $(s,-1):=-s$, and $|(s,1))|=|(s,-1)|=s$. For any $X=(x,p) \in s(M)$, we define sign$(X)=p$. Then $s(M)$ is a hypergroup with the addition given by
\begin{equation}\label{s(B)addition}
x+y = \left\{ \begin{array}{ll}
x & \textrm{if $|x|>|y|$ or $x=y$}\\
y & \textrm{if $|x|<|y|$ or $x=y$}\\
$[$-x,x$]$=\{(t,\pm 1)\mid t\leq |x|)\} & \textrm{if $y=-x$}
\end{array} \right.
\end{equation}
Let $s:M \longrightarrow s(M)$, $s \mapsto (s,1)$ be the associated map.\\
In \cite{jaiungthesis}, we showed that the symmetrization process can be generalized, in some cases, to semirings by imposing a coordinate-wise multiplication. The symmetrization of a totally ordered idempotent semiring is not a hyperring in general, however, we showed that the symmetrization of a totally ordered idempotent semifield is a hyperfield. The symmetrization is the main technique which we use in this section. Also, the essential idea is similar to that of J.Giansiracusa and N.Giansiracusa's construction of tropical schemes in \cite{noah}.\\

Let $\mathbb{R}_{max}:=\{\mathbb{R}\cup\{-\infty\},+,\max\}$ be the tropical semifield with a maximum convention. Recall that a tropical hypersurface $trop(V(f))$ defined by a polynomial equation $f \in \mathbb{R}_{max}[x_1,...,x_n]$ is the following set:
\begin{equation}\label{tropicalhypersurface}
\{a \in (\mathbb{R}_{max})^n \mid \textrm{ the maximum of $f$ is achieved at least twice} \}.
\end{equation} 
For the notational convenience, let $M=\mathbb{R}_{max}$ and $M_S=R=s(\mathbb{R}_{max})$, the symmetrization of $M$. Note that $M_S$ is a hyperfield since $M$ is a totally ordered idempotent semifield as we mentioned above.\\ 
Let $f(x_1,...,x_n) \in M[x_1,...,x_n]$. Write $f(x_1,...,x_n)=\sum_i m_i(x_1,...,x_n)$ as a sum of distinct monomials and fix this presentation. Then we define
\[f_{\hat{i}}(x_1,...,x_n):=\sum_{j \neq i} m_j(x_1,...,x_n) \in M[x_1,...,x_n] \quad \textrm{ for each $i$}.\]
By identifying an element $a \in M$ with the element $(a,1) \in M_S=R$, we define \[\tilde{f}_{\hat{i}}(x_1,...,x_n):=(\sum_{j \neq i} (m_j(x_1,...,x_n),1))+(m_i(x_1,...,x_n),-1) \in R[x_1,...,x_n].\]
With these notations we have the following description of a tropical hypersurface.
 
\begin{pro}\label{trposym}
With the same notation as above, we let
\[V(\tilde{f}_{\hat{i}}):=\{z \in R^n \mid 0_R \in \tilde{f}_{\hat{i}}(z)\},\quad  
HV(f):=\bigcup_i V(\tilde{f}_{\hat{i}}).\]
With the coordinate-wise symmetrization map, $\varphi=s^n:M^n \longrightarrow R^n$, we have a set bijection:
\[ trop(V(f)) \simeq (HV(f) \cap \Img(\varphi)),\] 
where $trop(V(f))$ is the tropical variety defined by $f \in M[x_1,...,x_n]$.
\end{pro}

\begin{rmk}
Even though we started by fixing one presentation of a polynomial equation $f \in M[x_1,...,x_n]$, the set $trop(V(f))$ does not depend on the chosen presentation of $f$. Therefore, even though $HV(f)$ may vary depending on a presentation of $f$, the set $HV(f)\bigcap \Img(\varphi)$ is invariant of the presentation as long as there is no repetition of monomials.
\end{rmk}
\noindent Before we prove Proposition \ref{trposym}, we present an example to show how this procedure works. 

\begin{myeg}\label{egnotation}
Let $f(x,y)=x+y+1 \in \mathbb{R}_{max}[x,y]$. Then $trop(V(f))$ consists of three rays:
\[
trop(V(f))=\{(x,y) \in \mathbb{R}_{max} \times \mathbb{R}_{max} \mid 1 \leq y=x, \textrm{ or } y \leq x=1, \textrm{ or }x \leq y=1\}.
\]
With the above notations, we have 
\begin{equation}\label{tropeg}
f_x(x,y):=y+1,\quad  \tilde{f}_x(x,y):=(1,1)y+(1,1)+(1,-1)x.
\end{equation} 
Since we only consider the `positive' solutions, $x$ and $y$ should be of the form $(t,1)$. Therefore, in this case, we may assume that
\begin{equation}\label{tropegequiv}
\tilde{f}_x(x,y)=(y,1)+(1,1)+(x,-1)=(y+1,1)+(x,-1).
\end{equation} 
 By the definition of symmetrization, we have
\[0_R \in \tilde{f}_x(x,y) \Longleftrightarrow y+1=x \textrm{ in } \mathbb{R}_{max}.\]
Thus, we obtain 
\[
\{(x,y) \in \mathbb{R}_{max} \times \mathbb{R}_{max} \mid 1\leq y=x, \textrm{ or } y \leq x=1\}=V(\tilde{f}_x(x,y))\bigcap (s(\mathbb{R}_{max}) \times s(\mathbb{R}_{max})).\] 
Similarly with $f_y(x,y)=x+1$ we have
\[ 0_R \in \tilde{f}_y(x,y) \Longleftrightarrow x+1=y \textrm{ in } \mathbb{R}_{max}.\]
This time, we obtain
\[
\{(x,y) \in \mathbb{R}_{max} \times \mathbb{R}_{max} \mid 1\leq x=y, \textrm{ or }x \leq y=1\}=V(\tilde{f}_y(x,y))\bigcap (s(\mathbb{R}_{max}) \times s(\mathbb{R}_{max})).\]  
Finally, with $f_1(x,y)=x+y$, we have
\[0_R \in \tilde{f}_1(x,y) \Longleftrightarrow x+y=1 \textrm{ in } \mathbb{R}_{max}.\]
This gives \[
\{(x,y) \in \mathbb{R}_{max} \times \mathbb{R}_{max} \mid y\leq x=1, \textrm{ or }x \leq y=1\}=V(\tilde{f}_1(x,y))\bigcap (s(\mathbb{R}_{max}) \times s(\mathbb{R}_{max})).\] 
By taking the union of all three we recover 
\[trop(V(f))=(\bigcup_{z \in \{x,y,1\}}V(\tilde{f}_z(x,y)))\bigcap(s(\mathbb{R}_{max}) \times s(\mathbb{R}_{max})).
\]
\end{myeg}

\noindent Now we give the proof of Proposition \ref{trposym}.

\begin{proof}
When $f$ is a single monomial, the result is clear since $0_M$ and $0_R$ will be the only solution for each. Thus we may assume that $f$ is not a monomial. If $z=(z_1,...,z_n) \in trop(V(f))$ then there exist $m_i(x_1,...,x_n)$, $m_j(x_1,...,x_n)$ with $i \neq j$ such that the value $m_i(z)=m_j(z)$ attains the maximum among all monomials $m_r(z)$. Then we have $f(z)=m_i(z)=m_j(z) \in M$. It follows that
\[0_R \in (f(z),1)+(m_i(z),-1)=(f_{\hat{i}}(z),1)+(m_i(z),-1)=\tilde{f}_{\hat{i}}(\varphi(z)).\] 
Thus we have $\varphi(z) \in HV(f)$.\\ 
Conversely, suppose that $\varphi(z) \in HV(f) \cap \Img(\varphi)$. Let $\varphi(z)=(\varphi(z_i))$, where $(z_i,1) \in M \times \{1\} \subseteq R$. Then, by the definition of $HV(f)=\bigcup V(\tilde{f}_{\hat{i}})$, $\varphi(z)$ is an element of $ V(\tilde{f}_{\hat{i}})$ for some $i$. In other words,
\[ 0_R\in (\sum_{j \neq i}(m_j(z),1)) +(m_i(z),-1)=\tilde{f}_{\hat{i}}(\varphi(z)). \]
Therefore, there exists some $r \neq i$ such that 
\[\sum_{j \neq i}(m_j(z),1)=(m_r(z),1)=(m_i(z),1)\quad 
\textrm{and} \quad m_j(z) \leq m_r(z)=m_i(z) \qquad \forall j \neq i,r.\] 
It follows that $z \in trop(V(f))$. So far we have showed that \[
\varphi(trop(V(f)))=HV(f) \cap \Img(\varphi).\] 
Since $\varphi$ is one-to-one, we conclude that $trop(V(f)) \simeq HV(f) \cap \Img(\varphi)$ as sets.
In other words, $trop(V(f))$ is the `positive' part of $HV(f)$.
\end{proof}

\begin{rmk}\label{tropdifferconventional}
Our definition \eqref{tropicalhypersurface} of $trop(V(f))$ may contain a point $a=(a_1,...,a_n)$ such that $a_i=-\infty(=0_M)$ for some $i$. This is little different from the conventional definition of a tropical hypersurface in which one excludes such points. However, from the proof of Proposition \ref{trposym}, one can observe that the subset of $trop(V(f))$ which does not have $0_M$ at any coordinate maps bijectively onto the subset of $H(V(f))\bigcap \Img(\varphi)$ which does not have $0_R$ at any coordinate. 
\end{rmk}
\noindent When $I$ is an ideal of $\mathbb{R}_{max}[x_1,..,x_n]$ one defines a tropical variety defined by $I$ as follows:
\begin{equation}\label{tropicalvar}
trop(V(I)):= \bigcap_{f\in I} trop(V(f)).
\end{equation}
One has to be careful with \eqref{tropicalvar} since the intersection is over all polynomials in $I$ not just over a set of generators of $I$ (cf. \cite{bernd}). To understand \eqref{tropicalvar} as the `positive' part of an algebraic set over hyperstructures, we extend the previous proposition as follows.

\begin{pro}\label{tropicalvareityispositivepart}
Let $I$ be an ideal of $\mathbb{R}_{max}[x_1,..,x_n]$. Then, with the same notation as Proposition \ref{trposym}, we have a set bijection via $\varphi=s^n$: 
\[trop(V(I)):= \bigcap_{f\in I} trop(V(f)) \simeq (\bigcap_{f\in I}HV(f)) \bigcap \Img\varphi.\] 
\end{pro}
\begin{proof}
Take any $z \in trop(V(I)) \subseteq (\mathbb{R}_{max})^n$, then by definition, $z \in \bigcap_{f\in I} trop(V(f))$. That is $z \in trop(V(f))$ $\forall f \in I$. It follows from the previous proposition that $\varphi(z) \in HV(f)$ $\forall f \in I$, thus $\varphi(z) \in (\bigcap_{f\in I}HV(f)) \bigcap \Img\varphi$.\\ 
Conversely, if $\varphi(z) \in (\bigcap_{f\in I}HV(f)) \bigcap \Img\varphi$ then $\varphi(z) \in HV(f)$ $\forall f \in I$. From the previous proposition it follows that $z \in trop(V(f))$ $\forall f \in I$, hence $z \in trop(V(I))$. Thus we have
\[\varphi(trop(V(I)))=(\bigcap_{f\in I}HV(f)) \bigcap \Img\varphi.\]
The conclusion follows from the injectivity of $\varphi$.
\end{proof}

We close this section by providing a pictorial image which shows how our construction works.\\
Let $\mathbb{R}_{\geq 0}$ be a semiring with a underlying set as the set of nonnegative real numbers, an addition is given by the maximum, and the multiplication is the usual multiplication of real numbers. Then one can observe that there is an isomorphism $\varphi$ of semirings as follows:
\begin{equation}\label{RmanxandTR}
\varphi: \mathbb{R}_{max} \longrightarrow \mathbb{R}_{\geq 0}, \quad r \mapsto t^r \textrm{ and } t^{-\infty}:=0, \textrm{ where t is any real number greater than $1$.}
\end{equation}
Under the symmetrization process, we have $s(\mathbb{R}_{\geq 0})=\mathcal{T}\mathbb{R}$, a Viro's hyperfield.

\begin{myeg}
Let $f(x,y):=(0\odot x) \oplus (0\odot y) \oplus 0 \in \mathbb{R}_{max}[x,y]$. Then, under the isomorphism $\varphi$, this becomes $f(x,y)=x+y+1 \in \mathcal{T}\mathbb{R}[x,y]$. If we follow the notation as in Example \ref{egnotation}, we obtain:
\[\tilde{f}_x(x,y)=y+1-x, \quad \tilde{f}_y(x,y)=x+1-y, \quad \tilde{f}_1(x,y)=x+y-1.
\]
First, we obtain the following picture from $f(x,y)$ in $\mathbb{R}^2$:
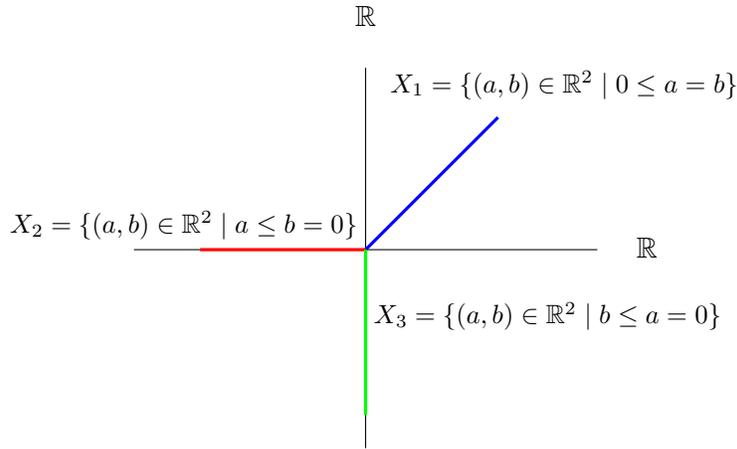
\begin{figure}[h]
\center{
\begin{tikzpicture}[>=angle 90,scale=2.2,text height=1.5ex, text depth=0.25ex]
\draw[-]
(-1.4,0) edge (1.4,0)
(0,-1.2) edge (0,1.1);
\node (y) at (0,1.4) {$\mathbb{R}$};
\node (x) at (1.7,0) {$\mathbb{R}$};
\node (l1) at (1.2,1) {\small{$X_1=\{(a,b) \in \mathbb{R}^2\mid 0\leq a=b$\}}};
\draw[-, blue, very thick]
(0,0) edge (0.8,0.8);
\node(l2) at (-1.1,0.15){\small{$X_2=\{(a,b) \in \mathbb{R}^2\mid a\leq b=0$\}}};
\draw[-, red, very thick]
(0,0) edge (-1,0);
\node(l3) at (1.1,-0.4){\small{$X_3=\{(a,b) \in \mathbb{R}^2\mid b\leq a=0$\}}};
\draw[-, green, very thick]
(0,0) edge (0,-1);
\end{tikzpicture}
}
\caption{Tropical Line in $\mathbb{R}^2$} 
\end{figure}
\\
On the other hand, we obtain the following picture in $(\mathcal{T}\mathbb{R})^2$:
\begin{figure}[h]
\center{
\begin{tikzpicture}[>=angle 90,scale=2.2,text height=1.5ex, text depth=0.25ex]
\draw[-]
(-5,0) edge (-3,0)
(-2.4,0) edge (-0.4,0)
(-4,-1) edge (-4,1)
(-1.4,-1) edge (-1.4,1)
(1.4,-1) edge (1.4,1)
(0.4,0) edge (2.4,0);
\node (fx) at (-4,1.6) {$\tilde{f}_x(x,y)=y+1-x$};
\node (fy) at (-1.4,1.6) {$\tilde{f}_y(x,y)=x+1-y$};
\node (f1) at (1.4,1.6) {$\tilde{f}_1(x,y)=x+y-1$};
\node (y) at (-4,1.2) {$\mathcal{T}\mathbb{R}$};
\node (y) at (-1.4,1.2) {$\mathcal{T}\mathbb{R}$};
\node (y) at (1.4,1.2) {$\mathcal{T}\mathbb{R}$};
\node (x) at (-2.8,0) {$\mathcal{T}\mathbb{R}$};
\node (x) at (-0.2,0) {$\mathcal{T}\mathbb{R}$};
\node (x) at (2.6,0) {$\mathcal{T}\mathbb{R}$};
\draw[-, blue, very thick]
(-0.7,0.5) edge (-0.2,1)
(-3.3,0.5) edge (-2.8,1);
\draw[-, red, very thick]
(1.4,0.5) edge (2.1,0.5)
(-4,0.5) edge (-3.3,0.5);
\draw[-, green, very thick]
(2.1,0.5) edge (2.1,0)
(-0.7,0) edge (-0.7,0.5);
\end{tikzpicture}
}
\caption{Tropical Line in $(\mathcal{T}\mathbb{R})^2$} 
\end{figure}
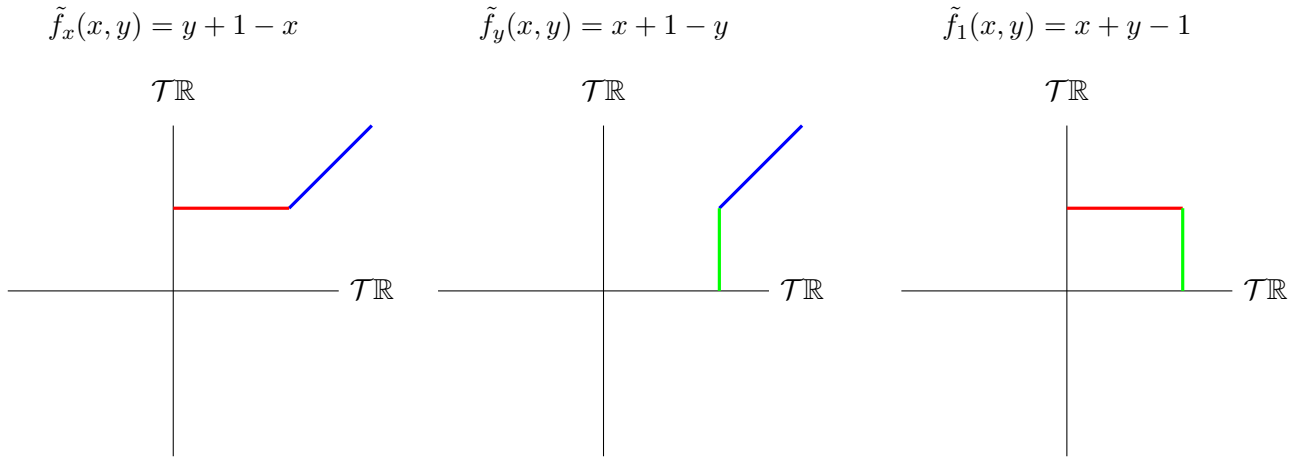

\end{myeg}
\subsection{Construction of hyperring schemes} \label{schmeandzeta}
In this subsection, we study several notions of algebraic geometry over hyperrings from the scheme-theoretic point of view: we use results in \S \ref{lem} to construct an integral hyperring scheme and prove that $\Gamma(X,\mathcal{O}_X) \simeq R$ for an affine integral hyperring scheme $(X,\mathcal{O}_X)$.\\ 
In classical algebraic geometry, a scheme is a pair $(X,\mathcal{O}_X)$ of a topological space $X$ and the structure sheaf $\mathcal{O}_X$ on $X$. The implementation of the notion of structure sheaf is essential to link local and global algebraic data.\\
Let $A$ be a commutative ring and $(X=\Spec A, \mathcal{O}_X)$ be an affine scheme. One of important results in classical algebraic geometry is the following:
\begin{equation}\label{globalsection}
\mathcal{O}_X(X) \simeq A.
\end{equation}
In other words, a commutative ring $A$ can be understood as the ring of global regular functions on the topological space $X=\Spec A$. When we directly generalize the construction of the structure sheaf of a commutative ring to a hyperring, \eqref{globalsection} no longer holds (see, Example \ref{sheafnoexample}). Furthermore, in this case, $\mathcal{O}_X$ does not even seem to be a sheaf of hyperrings (see, Remark \ref{structuresheafsubtle}). To this end, we construct the structure sheaf on the topological space $X=\Spec R$ only when $R$ is a hyperring without (multiplicative) zero-divisors. We follow the classical construction.\\

Let $R$ be a hyperring and $X=\Spec R$. For an open subset $U\subseteq X$, we define
\begin{equation}\label{structuresehafdefin}
 \mathcal{O}_X(U):=\{s:U \rightarrow \bigsqcup_{\mathfrak{p}\in U} R_\mathfrak{p}\},
\end{equation} 
where $s \in \mathcal{O}_X(U)$ are sections such that $s(\mathfrak{p}) \in R_\mathfrak{p}$ which also satisfy the following property: for each $\mathfrak{p} \in U$, there exist a neighborhood $V_\mathfrak{p}\subseteq U$ of $\mathfrak{p}$ and $a, f \in R$ such that
\begin{equation}\label{sheafcondition}
\forall \mathfrak{q} \in V_\mathfrak{p},\quad  f \notin \mathfrak{q} \textrm{ and } s(\mathfrak{q})=\frac{a}{f} \textrm{ in } R_\mathfrak{q}.
\end{equation}
A restriction map $\mathcal{O}_X(U) \longrightarrow \mathcal{O}_X(V)$ is given by sending $s$ to $s\circ i$, where $i:V \hookrightarrow U$ is an inclusion map. Then, clearly $\mathcal{O}_X$ is a sheaf of sets on $X$. Moreover, one can define the multiplication $s\cdot t$ of sections $s,t \in \mathcal{O}_X(U)$ as follows:
\begin{equation}\label{sheafmlti}
s \cdot t : U \rightarrow \bigsqcup R_\mathfrak{p}, \quad \mathfrak{p} \mapsto s(\mathfrak{p})t(\mathfrak{p}).
\end{equation}
Equipped with the above multiplication, one can easily see that $\mathcal{O}_X$ becomes a sheaf of (multiplicative) monoids on $X$. Furthermore, $\mathcal{O}_X(U)$ is equipped with the following hyperoperation:
\begin{equation}\label{sheafhypadd}
s+t=\{r \in \mathcal{O}_X(U) \mid r(\mathfrak{p}) \in s(\mathfrak{p})+t(\mathfrak{p}),\quad \forall \mathfrak{p} \in U \}.
\end{equation}

\begin{rmk}\label{structuresheafsubtle}
This construction is essentially same as in \cite{Rota}. However, in \cite{Rota}, the proof is incomplete in the sense that the authors did not prove that \eqref{sheafhypadd} is associative and distributive with respect to \eqref{sheafmlti}. Moreover, the main purpose of this subsection is to recover a hyperring $R$ as the hyperring of global sections on a topological space $\Spec R$ while the main goal of the authors of \cite{Rota} was to construct hyperring schemes. In Theorem \ref{sheaf}, we prove that when $R$ does not have (multiplicative) zero-divisors, $\mathcal{O}_X$ is indeed the sheaf of hyperrings, and $\mathcal{O}_X(X)\simeq R$.
\end{rmk}

\begin{mydef}
A hyperring $R$ is called a hyperdomain if $R$ does not have (multiplicative) zero-divisors. In other words, for $x, y \in R$, if $xy=0$ then either $x=0$ or $y=0$.
\end{mydef}

\noindent Let $R$ be a hyperdomain and $S:=R\backslash \{0\}$. Clearly $K:=\Frac(R)=S^{-1}R$ is a hyperfield and the canonical homomorphism $S^{-1}:R \longrightarrow K$ of hyperrings sending $r$ to $\frac{r}{1}$ is strict and injective.\\
Let $S_f=\{1,f\neq 0,...,f^n,...\}$ be a multiplicative subset of $R$ and $R_f:=S_f^{-1}R$, then we have the canonical homomorphisms of hyperrings $R \hookrightarrow R_f \hookrightarrow K$ which are injective and strict. Therefore, in the sequel, we consider $R_f$ as a hyperring extension of $R$ and $K$ as a hyperring extension of both $R$ and $R_f$ via the above canonical maps. For $\mathfrak{p} \in \Spec R$, we denote by $R_\mathfrak{p}$ the hyperring $S^{-1}R$, where $S=R\backslash \mathfrak{p}$.

\begin{lem}\label{setwithhyperaddtionisotohyperring}
Let $A$ be a set equipped with the two binary operations: \[+_A:A \times A \longrightarrow P^*(A),\quad \cdot_A:A\times A \longrightarrow A ,\]
where $P^*(A)$ is the set of nonempty subsets of $A$. Suppose that $R$ is a hyperring and there exists a set bijection $\varphi:A \longrightarrow R$ such that $\varphi(a+_Ab)=\varphi(a)+\varphi(b)$ and $\varphi(a\cdot_Ab)=\varphi(a)\varphi(b)$. Then, $A$ is a hyperring isomorphic to $R$.
\end{lem}
\begin{proof}
The proof is straightforward. For example, $\varphi^{-1}(0_R):=0_A$ is the neutral element. In fact, for $a \in A$, we have $\varphi(0_A+_Aa)=\varphi(0_A)+\varphi(a)=0_R+\varphi(a)=\varphi(a)$. Since $\varphi$ is bijective, it follows that $0_A+_Aa=a$ $\forall a\in A$. Similarly, $1_A:=\varphi^{-1}(1_R)$ is the identity element. For $a \in A$, we can write $-\varphi(a)=\varphi(b)$ for some $b \in A$. Then, we have $\varphi(a+_Ab)=\varphi(a)+\varphi(b)=\varphi(a)-\varphi(a)$. It follows that $0_R \in \varphi(a+_Ab)$, hence $0_A \in a+_Ab$. The other properties can be easily checked and hence $A$ is a hyperring. Clearly, $A$ and $R$ are isomorphic via $\varphi$.
\end{proof}

\begin{mythm}\label{sheaf}
Let $R$ be a hyperdomain, $K=\Frac(R)$, and $X=\Spec R$. Let $\mathcal{O}_X$ be the sheaf of multiplicative monoids on $X$ as in \eqref{structuresehafdefin}, equipped with the hyperaddition \eqref{sheafhypadd}. Then, the following hold
\begin{enumerate}
\item
$\mathcal{O}_X(D(f))$ is a hyperring isomorphic to $R_f$. In particular, if $f=1$, we have 
\[R \simeq \mathcal{O}_X(X)(=\Gamma(X,\mathcal{O}_X)).\]
\item
For each open subset $U$ of $X$, $\mathcal{O}_X(U)$ is a hyperring. More precisely, $\mathcal{O}_X(U)$ is isomorphic to the following hyperring:
\[\mathcal{O}_X(U) \simeq Y(U):=\{u \in K\mid\forall \mathfrak{p} \in U,u=\frac{a}{b}\textrm{ for some } b \notin \mathfrak{p}\}.\]
Moreover, by considering the canonical map $R_f \hookrightarrow K$, we have
\[\mathcal{O}_X(U)\simeq \bigcap_{D(f)\subseteq U}\mathcal{O}_X(D(f)).\]
\item
For each $\mathfrak{p} \in X$, the stalk $\mathcal{O}_{X,\mathfrak{p}}$ exists and is isomorphic to $R_\mathfrak{p}$.
\end{enumerate}
\end{mythm}

\begin{proof}
The proof is similar to the classical case (cf. \cite{Har}). We frequently use the results in \S \ref{lem}, in particular, Hilbert's Nullstellensatz of hyperrings (Lemma \ref{nilradical}).
\begin{enumerate}
\item
Consider the following map:
\begin{equation}\label{sheafproof1}
\psi: R_f \rightarrow \mathcal{O}_X(D(f)), \quad \frac{a}{f^n} \mapsto s, \textrm{ where } s(\mathfrak{p})=\frac{a}{f^n} \textrm{ in } R_\mathfrak{p}.
\end{equation}
Clearly, $\psi$ is well defined since the map $s$ defined as in \eqref{sheafproof1} satisfies the condition \eqref{sheafcondition}. It also follows from the definition that
\[\psi(\frac{a}{f^n}\cdot \frac{b}{f^m})=\psi(\frac{a}{f^n}) \cdot \psi(\frac{b}{f^m}), \quad \psi(\frac{a}{f^n} + \frac{b}{f^m}) \subseteq \psi(\frac{a}{f^n}) + \psi(\frac{b}{f^m}).\]
First, we claim that $\psi$ is one-to-one. Indeed, suppose that $\psi(\frac{a}{f^n})=\psi(\frac{b}{f^m})$. Then, $\frac{a}{f^n}=\frac{b}{f^m}$ as elements of $R_\mathfrak{p}$ $\forall \mathfrak{p} \in D(f)$.
Hence there is an element $h \not\in \mathfrak{p}$ such that $hf^ma=hf^nb$ in $R$. This implies that $0 \in hf^ma-hf^nb=h(f^ma-f^nb)$. However, since  $h \not\in \mathfrak{p}$ (hence,  $h \neq 0$) and $R$ is a hyperdomain, it follows that $f^ma=f^nb$. This implies that $\frac{a}{f^n}=\frac{b}{f^m}$ in $R_f$, thus $\psi$ is one-to-one.\\
Next, we claim that $\psi$ is onto. Take $s \in \mathcal{O}_X(D(f))$. Then, we can cover $D(f)$ with open sets $V_i$ so that $s$ is represented by a quotient $\frac{a_i}{g_i}$ on $V_i$ with $g_i \not\in \mathfrak{p}$ $\forall \mathfrak{p} \in V_i$ from \eqref{sheafcondition}. Since open subsets of the form $D(h)$ form a basis, we may assume that $V_i=D(h_i)$ for some $h_i \in R$. Let $(h_i)$ and $(g_i)$ be the hyperideals generated by $h_i$ and $g_i$. Since $s$ is represented by $\frac{a_i}{g_i}$ on $D(h_i)$, we have $D(h_i) \subseteq D(g_i)$; hence, $V((h_i)) \supseteq V((g_i))$. It follows from Lemma \ref{nilradical} that $\sqrt{(h_i)} \subseteq \sqrt{(g_i)}$ and hence $h_i^n \in (g_i)$ for some $n \in \mathbb{N}$. Then, from Lemma \ref{generatingideal}, we have $h_i^n=cg_i$ for some $c \in R$. Hence $\frac{a_i}{g_i}=\frac{ca_i}{h_i^n}$. If we replace $h_i$ by $h_i^n$ (since $D(h_i)=D(h_i^n)$) and $a_i$ by $ca_i$, we may assume that $D(f)$ is covered by the open subsets $D(h_i)$ on which $s$ is represented by $\frac{a_i}{h_i}$. Moreover, as in the classical case, we observe that $D(f)$ can be covered by finitely many $D(h_i)$. In fact,
\begin{equation}\label{d(f)openclosed}
D(f) \subseteq \bigcup D(h_i) \Longleftrightarrow V((f)) \supseteq \bigcap V((h_i)). 
\end{equation}
Let $I_i=(h_i)$,  $I=<I_i>$, and $J=(f)$. Then, \eqref{d(f)openclosed} can be written as follows:
\begin{equation}\label{d(f)below}
\bigcap V(I_i)=V(I), \quad D(f)\subseteq \bigcup D(h_i) \Longleftrightarrow V(J) \supseteq V(I).
\end{equation}
It follows from Lemma \ref{nilradical} that $\sqrt{J} \subseteq \sqrt{I}$, thus $f^n \in I$ for some $n \in \mathbb{N}$. Then, from Lemma \ref{generatingideal}, we have
\begin{equation}\label{finitecover}
f^n \in \sum_{i=1}^r b_ih_i \textrm{ for some } b_i \in R.
\end{equation}
We claim that $D(f)$ can be covered by $D(h_1) \cup ...\cup D(h_r)$. Indeed, this is equivalent to
\[V((f)) \supseteq \bigcap_{i=1}^r V((h_i))=V(<(h_i)>_{i=1,...,r}).\]
Let $I:=<h_i>_{i=1,...,r}$. Suppose that $\mathfrak{p} \in V(I)$. Since $I\subseteq \mathfrak{p}$, it follows from \eqref{finitecover} that $f^n \in \mathfrak{p}$, hence $f \in \mathfrak{p}$. This implies that $(f) \subseteq \mathfrak{p}$, thus $\mathfrak{p} \in V((f))$. From now on, we fix the elements $h_1,...,h_r$ such that $D(f) \subseteq D(h_1) \cup ... \cup D(h_r)$. Then, on $D(h_ih_j)$, we have two elements $\frac{a_i}{h_i}$, $\frac{a_j}{h_j}$ of $R_{h_ih_j}$ which can be considered as (the restriction of) the same element $s$. It follows from the injectivity of $\psi$ (on $D(h_ih_j)$) that one should have $\frac{a_i}{h_i}=\frac{a_j}{h_j}$ in $R_{h_ih_j}$. Therefore, $(h_ih_j)^nh_ja_i=(h_ih_j)^nh_ia_j$ for some $n \in \mathbb{N}$. However, since $R$ is a hyperdomain, we have $(h_ih_j)^n \neq 0$. It follows that $h_ja_i=h_ia_j$ $\forall i,j=1,...,r$ from the uniqueness of an additive inverse.\\
Let $f^n \in \sum_{i=1}^r b_ih_i$ as in \eqref{finitecover}. Then, for each $j\in\{1,...,r\}$, we have
\[f^na_j \in (\sum_i b_ih_i)a_j=\sum_i b_i a_jh_i=\sum_i b_ia_ih_j=(\sum_ib_ia_i)h_j.\]
It follows that for each $j=1,...,r$, there exists $\beta_j \in \sum_i b_ia_i$ such that $f^na_j=\beta_jh_j$. Hence, we have
\begin{equation}\label{onh_ih_j}
\frac{\beta_j}{f^n}=\frac{a_j}{h_j} \quad \textrm{ on } D(h_j).
\end{equation}
We claim that $\beta_i=\beta_j$ $\forall i,j=1,...,r$. Indeed, on $D(h_ih_j)$, we proved that $\frac{a_j}{h_j}=\frac{a_i}{h_i}$. Together with \eqref{onh_ih_j} and the injectivity of $\psi$, we have
\[ \frac{\beta_j}{f^n}=\frac{a_j}{h_j}=\frac{a_i}{h_i}= \frac{\beta_i}{f^n}  \textrm{ on } D(h_ih_j).\]
Therefore, $\exists m \in \mathbb{N}$ such that $(h_ih_j)^mf^n\beta_j=(h_ih_j)^mf^n\beta_i$.
Equivalently, we have $0 \in (h_ih_j)^mf^n(\beta_j-\beta_i)$. However, we know that $(h_ih_j)^m, f^n \neq 0$ since $h_ih_j, f \neq 0$ and $R$ is a hyperdomain. It follows that $0 \in \beta_i - \beta_j$, thus $\beta_i=\beta_j$ from the uniqueness of an additive inverse. Let $\beta$ be this common value $\beta_i$. Then, we have $f^na_j=\beta h_j$ $\forall j=1,...,r$. Therefore, $\frac{\beta}{f^n}=\frac{a_j}{h_j} \textrm{ on } D(h_j)$. In other words, $\psi(\frac{\beta}{f^n})=s$. This shows that $\psi$ is onto. This, however, does not complete the proof. We need to show that $\psi(ab)=\psi(a)\psi(b)$ and $\psi(a+b)=\psi(a)+\psi(b)$, then the result follows from Lemma \ref{setwithhyperaddtionisotohyperring}. Clearly, we have $\psi(ab)=\psi(a)\psi(b)$ and $ \psi(\frac{a}{f^n}+\frac{b}{f^m}) \subseteq \psi(\frac{a}{f^n})+\psi(\frac{b}{f^m})$.
We show the following:
\[ \psi(\frac{a}{f^n})+\psi(\frac{b}{f^m})\subseteq\psi(\frac{a}{f^n}+\frac{b}{f^m}). \]
Let $s=\psi(\frac{a}{f^n})$, $t=\psi(\frac{b}{f^m})$. Then, we have
\[ s+t=\{r \in \mathcal{O}_X(D(f)) \mid r(\mathfrak{p}) \in s(\mathfrak{p})+t(\mathfrak{p})\quad \forall \mathfrak{p} \in D(f) \}.\]
For $r \in s+t$, since $\psi$ is onto, $r=\psi(\frac{c}{f^l})$ for some $\frac{c}{f^l} \in R_f$.
It follows from $r(\mathfrak{p}) \in s(\mathfrak{p}) + t(\mathfrak{p})$ that $\frac{c}{f^l} \in \frac{a}{f^n} + \frac{b}{f^m}$ in $R_\mathfrak{p}$, and this is equivalent to the following:
\[\frac{c}{f^l}=\frac{d}{f^{n+m}} \textrm{ for some } d \in (af^m+bf^n) \textrm{ in } R_\mathfrak{p}.\]
Therefore, $ucf^{n+m}=udf^l$ for some $u \in R \backslash \mathfrak{p}$. Since $u \neq 0$, we have $cf^{n+m}=df^l$. Equivalently, $\frac{c}{f^l}=\frac{d}{f^{n+m}}$ in $R_f$. However, $\frac{d}{f^{n+m}} \in \frac{a}{f^n}+\frac{b}{f^m}$, therefore $s+t \subseteq \psi(\frac{a}{f^n}+\frac{b}{f^m})$. This shows the other inclusion. The conclusion follows from Lemma \ref{setwithhyperaddtionisotohyperring}.
\item
One can easily see that $Y(U)$ is a hyperring (in fact, a sub-hyperring of $K$). We show that there exists a bijection $\varphi$ of sets from $\mathcal{O}_X(U)$ to $Y(U)$ such that $\varphi(a+b)=\varphi(a)+\varphi(b)$ and $\varphi(ab)=\varphi(a)\varphi(b)$. Then, the first assertion will follow from Lemma \ref{setwithhyperaddtionisotohyperring}. Indeed, if $s \in \mathcal{O}_X(U)$, then from the same argument (in the proof of $1$), we can find a cover $U=\bigcup D(h_i)$ such that
$s=\frac{a_i}{h_i}$ on $D(h_i)$. However, since $R$ has no (multiplicative) zero-divisor, $X=\Spec R$ is irreducible from Proposition \ref{irreducible}. Thus, $D(h_i) \cap D(h_j) \neq \emptyset$ $\forall i,j$. It follows that $\frac{a_i}{h_i}=\frac{a_j}{h_j}$ on $D(h_i) \cap D(h_j)$,  
equivalently $0 \in s_{ij}(a_ih_j - h_ia_j)$ for some $s_{ij} \neq 0 \in R$. Since $s_{ij} \neq 0$ and R is a hyperdomain, it follows that
\[0 \in (a_ih_j - h_ia_j) \Longleftrightarrow a_ih_j=a_jh_i \Longleftrightarrow \frac{a_i}{h_i}=\frac{a_j}{h_j} \textrm{ as elements of } K=\Frac (R).\]
Let $u=\frac{a}{b}$ be this common value in $K$. Then, for each $\mathfrak{p} \in U$, we have $\mathfrak{p} \in D(h_i)$ for some $h_i \not \in \mathfrak{p}$ and $\frac{a}{b}=\frac{a_i}{h_i}$ on $D(h_i)$. It follows that $u \in Y(U)$. Since $\mathcal{O}_X$ is a sheaf, $u \in Y(U)$ is uniquely determined by $s$. We let $\varphi(s):=u$. Then, we have 
\[ \varphi: \mathcal{O}_X(U) \longrightarrow Y(U):=\{u \in K\mid \forall \mathfrak{p} \in U, u=\frac{a}{b}, b \not \in \mathfrak{p}\} \subseteq K.\]
$\varphi$ is well defined and one-to-one since $\mathcal{O}_X$ is a sheaf (of sets). We claim that $\varphi$ is onto. In fact, for $u=\frac{a}{b} \in Y(U)$, we define $s:U \longrightarrow \bigsqcup_{\mathfrak{p} \in U}R_\mathfrak{p}$ such that $s(\mathfrak{p})=\frac{a}{b}=\frac{a'}{b'}$ for $b' \not \in \mathfrak{p}$ from the definition of $Y(U)$. Then $s \in \mathcal{O}_X(U)$. Next, it follows from the definition that $\varphi(s\cdot t)=\varphi(s)\cdot \varphi(t)$. Furthermore, we have $\varphi(s+t) \subseteq \varphi(s)+\varphi(t)$.
Indeed, we have
\begin{equation}\label{ains+t}
\alpha \in s+t \Longleftrightarrow \alpha(\mathfrak{p}) \in s(\mathfrak{p})+t(\mathfrak{p}) \quad \forall \mathfrak{p} \in U.
\end{equation}
However, since $\varphi$ is bijective, each section is globally represented by an element of $K$. Suppose that $\alpha, s, t$ are globally represented by $\frac{g}{f},\frac{a}{h},\frac{b}{m}$ respectively. Then, \eqref{ains+t} is equivalent to the following:
\[\alpha(\mathfrak{p}) \in s(\mathfrak{p})+t(\mathfrak{p}) \Longleftrightarrow \frac{g}{f} \in \frac{a}{h}+\frac{b}{m} \Longleftrightarrow \varphi(\alpha) \in \varphi(s)+\varphi(t).\]
Conversely, for $\frac{g}{f} \in \varphi(s)+\varphi(t)=\frac{a}{h}+\frac{b}{m}$, we have $\alpha \in \mathcal{O}_X(U)$ such that $\alpha(\mathfrak{p})=\frac{g}{f}$ at $R_\mathfrak{p}$. It follows that $\alpha \in s+t$, and $\varphi(s)+\varphi(t) \subseteq \varphi(s+t)$.
This shows that $\mathcal{O}_X(U)$ is isomorphic to $Y(U)$ via $\varphi$ from Lemma \ref{setwithhyperaddtionisotohyperring}.\\ 
Next, we prove the second assertion. For $D(f) \subseteq U$, we have $Y(U) \subseteq Y(D(f)) \subseteq K$.
This implies that $Y(U) \subseteq \bigcap_{D(f) \subseteq U}Y(D(f))$.
Conversely, suppose that
$u=\frac{a}{b} \in \bigcap_{D(f)\subseteq U} Y(D(f))$. Then, for each $\mathfrak{p} \in U$, we have $\mathfrak{p} \in D(f)$ for some $D(f)\subseteq U$. Since $u \in Y(D(f))$, $u$ can be written as $\frac{x}{y}$ so that $y \not \in  \mathfrak{p}$. It follows that $u \in Y(U)$, and $Y(U)=\bigcap_{D(f)\subseteq U}Y(D(f))$. One observes that $Y(D(f))=R_f \subseteq K$. Thus, under $\mathcal{O}_X(D(f))\simeq R_f$, we have $\mathcal{O}_X(U)\simeq \bigcap_{D(f)\subseteq U}\mathcal{O}_X(D(f))$.
\item
In general, direct limits do not exist in the category of hyperrings. Thus, one should be careful. Since open sets of the form $D(f)$ form a basis of $X$, it is enough to show that 
\begin{equation}\label{stakpro}
\varinjlim_{D(f) \ni \mathfrak{p}} \mathcal{O}_X(D(f))=R_\mathfrak{p}.
\end{equation}
For each $f \in R$, let $\psi_f:\mathcal{O}_X(D(f)) \longrightarrow R_\mathfrak{p}$, $s \mapsto s(\mathfrak{p})$. Then, we have the following commutative diagram:
\[
\xymatrixcolsep{3pc}\xymatrix{
\mathcal{O}_X(D(f)) \ar@{->}[rd]_{\psi_f} \ar@{->}[rr]^{\rho}
&
&\mathcal{O}_X(D(g)) \ar@{->}[ld]^{\psi_g}\\
&R_\mathfrak{p}},
\]
where $\rho$ is a restriction map of the structure sheaf $\mathcal{O}_X$. Let $H$ be a hyperring and suppose that we have another commutative diagram:
\[
\xymatrixcolsep{3pc}\xymatrix{
\mathcal{O}_X(D(f)) \ar@{->}[rd]^{\psi_f} \ar@{->}[rr]^{\rho}
\ar@{->}[rdd]_{\varphi_f}
&
&\mathcal{O}_X(D(g)) \ar@{->}[ld]_{\psi_g} \ar@{->}[ldd]^{\varphi_g}\\
&R_\mathfrak{p}\\
&H}.
\] 
By considering $R_\mathfrak{p}$ as a sub-hyperring of $K$, let us define the map $\psi$ as follows: 
\[\psi:R_\mathfrak{p} \longrightarrow H, \quad \frac{b}{t} \mapsto \varphi_t(\frac{b}{t}),\]
where $\frac{b}{t}$ is considered as an element of $\mathcal{O}_X(D(t))$ such that $\frac{b}{t}(\mathfrak{q})=\frac{b}{t}$ in $R_\mathfrak{q}$ for each $\mathfrak{q} \in D(t)$. If $\psi$ is a well-defined homomorphism of hyperrings, the uniqueness of such map easily follows. Indeed, suppose that $\varphi:R_\mathfrak{p} \longrightarrow H$ is the homomorphism of hyperrings such that $\varphi_f=\varphi\circ \psi_f$ $\forall f \in R$ for $\mathfrak{p} \in D(f)$. A section $s$ of $\mathcal{O}_X(D(f))$ is represented by $\frac{b}{f^n}$ (from the first part of the proposition). Therefore, $\psi_f(s)=\psi_f(\frac{b}{f^n})$, and $\varphi_f(s)=\varphi\circ\psi_f(s)=\varphi\circ\psi_f(\frac{b}{f^n})=\varphi(\frac{b}{f^n})$. However, we have $\varphi_f(s)=\psi \circ \psi_f(s)=\psi(\frac{b}{f^n})$. Thus, such $\psi$ is unique if it exists.\\ 
Next, we show that $\psi$ is well defined. Indeed, if $\frac{b}{t}=\frac{b'}{t'}$, then we have $bt'=b't$ since $R$ is a hyperdomain. It follows that $D(bt')=D(b't) \subseteq D(t)$, and we have the following commutative diagram: 
\[
\xymatrixcolsep{3pc}\xymatrix{
\mathcal{O}_X(D(t)) \ar@{->}[rd]_{\varphi_t} \ar@{->}[rr]^{\rho}
&
&\mathcal{O}_X(D(b't))=\mathcal{O}_X(D(bt')) \ar@{->}[ld]^{\varphi_{b't}}\\
&H}.
\]
From the similar argument with $t'$, we have $\varphi_t(\frac{b}{t})=\varphi_{t'}(\frac{b'}{t'})$. This shows that $\psi$ does not depend on the choice of $t$, hence $\psi$ is well defined. For $\frac{a}{f}$, $\frac{b}{g} \in R_\mathfrak{p}$, by considering $\varphi_{fg}$, we have $\psi(\frac{a}{f})=\varphi_{fg}(\frac{a}{f})$, $\psi(\frac{b}{g})=\varphi_{fg}(\frac{b}{g})$, and $\psi(\frac{ab}{fg})=\varphi_{fg}(\frac{ab}{fg})$. Thus, $\psi(\frac{a}{f}\frac{b}{g})=\psi(\frac{a}{f})\psi(\frac{b}{g})$. Similarly, we have $\psi(\frac{a}{f}+\frac{b}{g}) \subseteq \psi(\frac{a}{f})+\psi(\frac{b}{g})$. Hence, $\psi$ is a homomorphism of hyperrings. Finally, since $\{D(f)\}$ is a basis of $X$, we have
\[\varinjlim_{U \ni \mathfrak{p}}\mathcal{O}_X(U)=\varinjlim_{D(f) \ni \mathfrak{p}} \mathcal{O}_X(D(f)).\]
Therefore, we conclude that $\mathcal{O}_{X,\mathfrak{p}}\simeq R_\mathfrak{p}$.
\end{enumerate}
\end{proof}

\noindent When $R$ is a hyperdomain, we call the pair $(X=\Spec R, \mathcal{O}_X)$ as in Theorem \ref{sheaf} an integral affine hyperring scheme. The following example shows that if $R$ has zero divisors, then in general $R \neq \Gamma(X,\mathcal{O}_X)$.

\begin{myeg}\label{sheafnoexample}
Consider the following quotient hyperring $R$:
\[ R=\mathbb{Q} \oplus \mathbb{Q} /G, \textrm{ where } G=\{(1,1),(-1,-1)\}.\]
Then, $\Spec R=\{\mathfrak{p}_1,\mathfrak{p}_2\}$ with $\mathfrak{p}_1=<[(1,0)]>$ and $\mathfrak{p}_2=<[(0,1)]>$. Each $\mathfrak{p}_j$ becomes one point open and closed subset of $X$ and the intersection $\mathfrak{p}_1 \cap \mathfrak{p}_2$ is empty. Furthermore, one can easily check that $R_{\mathfrak{p}_i} \simeq \mathbb{Q}/H$, where $H=\{1,-1\}$. Therefore,
\[\Gamma(X,\mathcal{O}_X)\simeq (\mathbb{Q}/H) \oplus (\mathbb{Q}/H) \neq R.\]
\end{myeg}

\begin{rmk}
\begin{enumerate}
\item
One can construct other examples of affine hyperring schemes $X=\Spec R$ for which $R \neq \Gamma(X,\mathcal{O}_X)$, but all such examples are disconnected. We do not have yet any example of a connected topological space $X=\Spec R$ with $R \neq \Gamma(X,\mathcal{O}_X)$. On the other hand, being connected is not a necessary condition. In fact, let $A=\mathbb{Z}/12\mathbb{Z}$ and $G=\{1,5\} \subseteq (\mathbb{Z}/12\mathbb{Z})^{\times}$. Then, with the quotient hyperring $R=A/G$, the space $X=\Spec R$ is disconnected (consist of two points), however, one can easily check that $R \simeq \Gamma(X,\mathcal{O}_X)$.
\item
What seems interesting is that by enhancing a multiplicative condition (hyperdomain) we could fix an additive weakness (multi-valued addition). In fact, similar observation has been made by O.Lorscheid in \cite{oliver1} stating that in the classical construction of schemes, a multiplicative structure plays a major role while an additive structures is less important. 
\end{enumerate}
\end{rmk}

\noindent Let $R$ be a hyperring, $X=\Spec R$, and $\mathcal{O}_X$ be the structure sheaf (of sets) of $X$. Then, as we previously mentioned in Remark \ref{structuresheafsubtle}, $\Gamma(X,\mathcal{O}_X)$ does not have to be a hyperring. Moreover, even if $\Gamma(X,\mathcal{O}_X)$ is a hyperring, Example \ref{sheafnoexample} shows that the natural map $R\longrightarrow \Gamma(X,\mathcal{O}_X)$ is not even an injective map in general. By appealing to the classical construction of Cartier divisors, we define the presheaf $\mathcal{F}_X$ of hyperrings on $X=\Spec R$ which slightly generalizes $\mathcal{O}_X$ (cf. Remark \ref{OFsame} and Proposition \ref{Fgen}).\\
Let $S:=\{\alpha \in R\mid \alpha\textrm{ is not a zero-divisor}\}$. In other words, $S$ is the set of regular elements of $R$. Then, $S\neq \emptyset$ since $1 \in S$. Furthermore, $S$ is a multiplicative subset of $R$, therefore one can define $K:=S^{-1}R$. In what follows, we denote by $R$ a hyperring, and $S$, $K$ as above. Note that by a sub-hyperring $R$ of a hyperring $L$ we mean a subset $R$ of $L$ which is a hyperring with the induced operations.

\begin{lem}\label{regulareltebd}
Let $S$ be the set as above, and $f \in S$. Let $\varphi:R \longrightarrow R_f$, $\varphi(a)=\frac{a}{1}$ be the natural map of localization and $\psi:R_f \longrightarrow K:=S^{-1}R$ be a homomorphism of hyperrings such that $\psi(\frac{a}{f^n})=\frac{a}{f^n}$. Then $\varphi$ and $\psi$ are strict, injective homomorphisms of hyperrings.
\end{lem}
\begin{proof}
We only prove the case of $\varphi$ since the proof for $\psi$ is similar. If $\frac{a}{1}=\frac{b}{1}$ then $f^na=f^nb$ for some $n \in \mathbb{N}$. This implies that $0 \in f^n(a-b)$. Hence, we have $c \in (a-b)$ such that $0=f^nc$. Since $f^n \in S$ can not be a zero-divisor, we have $c=0$, therefore $a=b$. This shows that $\varphi$ is injective. Furthermore, if $\gamma \in \varphi(a)+\varphi(b)=\frac{a}{1}+\frac{b}{1}$, then $\gamma=\frac{t}{1}$ for some $t \in a+b$. Therefore $\gamma=\varphi(t)$, and $\varphi$ is strict.
\end{proof}

\noindent For each open subset $U$ of $X=\Spec R$, we define the following set:
\begin{equation}\label{definofF}
\mathcal{F}_X(U):=\{u \in K\mid \forall \mathfrak{p} \in U, u=\frac{a}{b}, b \in S\cap \mathfrak{p}^c\}.
\end{equation} 
In other words, $u \in K$ is an element of $\mathcal{F}_X(U)$ if $u$ has a representative $\frac{a}{b}$ such that $b \not \in \mathfrak{p}$ for each $\mathfrak{p} \in U$. The restriction map is given by the natural injection. i.e. if $V \subseteq U$, then we have $\mathcal{F}_X(U) \hookrightarrow \mathcal{F}_X(V)$. Then, one can easily observe that $\mathcal{F}_X(U)$ is a hyperring. Thus, $\mathcal{F}_X$ becomes a presheaf of hyperrings on $X=\Spec R$.

\begin{rmk}\label{OFsame}
It follows from Theorem \ref{sheaf} that when $R$ is a hyperdomain, we have $\mathcal{O}_X(U) \simeq \mathcal{F}_X(U)$ for each open subset $U$ of $X$. Therefore, in this case, $\mathcal{F}_X$ is indeed a sheaf of hyperring and $\mathcal{F}_X(X)=R$.
\end{rmk}

\begin{pro} \label{sheaffirre}
Let $R$ be a hyperring. If $X=\Spec R$ is irreducible, then $\mathcal{F}_X$ is a sheaf of hyperrings.
\end{pro}
\begin{proof}
Since $\mathcal{F}_X(U)$ is clearly a hyperring, we only have to prove that $\mathcal{F}_X$ is a sheaf. Suppose that $U=\bigcup V_i$ is an open covering of $U$. Firstly, if $s \in \mathcal{F}_X(U)$ is an element such that $s|_{V_i}=0$ for all $i$, then we have to show that $s=0$. However, this is clear since the restriction map is injective. Secondly, let $s_i \in \mathcal{F}_X(V_i)$ such that $s_i|_{V_i \cap V_j}=s_j|_{V_i \cap V_j}$ for all $i,j$. Since $X$ is irreducible, it follows that $V_i \cap V_j \neq \emptyset$ $\forall i,j$. Moreover, the condition $s_i|_{V_i \cap V_j}=s_j|_{V_i \cap V_j}$ means that $s_i=s_j$ as elements of $K$. Let $s$ be this common element of $K$. Then, $\{s_i\}$ can be glued to $s$. Clearly, $s$ is an element of $\mathcal{F}_X(U)$.
\end{proof}

\noindent The following proposition shows that $\mathcal{F}_X$ behaves more nicely than $\mathcal{O}_X$ in some cases. 

\begin{pro}\label{Fgen}
Let $R$ be a hyperring and assume that $X=\Spec R$ is irreducible. Then, for $f \in S$, there exists a canonical injective and strict homomorphism $\varphi:R_f \longrightarrow \mathcal{F}_X(D(f))$. In particular, $R$ is a sub-hyperring of $\mathcal{F}_X(X)$. Furthermore, if $R$ has a unique maximal hyperideal, then $R\simeq \mathcal{F}_X(X)$.
\end{pro}
\begin{proof}
From Lemma \ref{regulareltebd}, there exists a canonical injective and strict homomorphism $\psi: R_f \longrightarrow K$. From the definition of $\mathcal{F}_X(D(f))$, one sees that the image of $\psi$ lies in $\mathcal{F}_X(D(f)) \subseteq K$. Therefore, $\psi$ becomes our desired $\varphi$. When $R$ has a unique maximal hyperideal, we have to show that any element $u$ of $\mathcal{F}_X(X)$ is of the form $\frac{a}{1}$ for some $a \in R$. Suppose that $\mathfrak{m}$ is the maximal ideal of $R$. Then, $u \in \mathcal{F}_X(X)$ means that $u=\frac{a}{b}$ for some $b \in S-\mathfrak{m}$. Since $\mathfrak{m}$ is the only maximal hyperideal of $R$, it follows from Lemma \ref{generatingideal} that $1 \in a+b$ for some $a \in \mathfrak{m}$. Therefore, $b \in 1-a$ and $b$ is a unit by Lemma \ref{jacobsonradical}. Thus, $u=\frac{a}{b}=\frac{ab^{-1}}{1}$ and $R\simeq \mathcal{F}_X(X)$.
\end{proof}

\noindent Next, we prove that the opposite category of hyperdomains and the category of integral affine hyperring schemes are equivalent via the contravariant functors, $\Spec$ and $\Gamma$. Note that one can directly generalize the notion of a ringed space to define a hyperringed space. However, the notion of a locally hyperringed space should be treated with greater care since the category of hyperrings does not have (co)limits in general. Nevertheless, an integral affine hyperring scheme $(X,\mathcal{O}_X)$ can be considered as a locally hyperringed space thanks to Theorem \ref{sheaf}. Thus, in what follows we consider $(X,\mathcal{O}_X)$ as a locally hyperringed space in the sense of the direct generalization of the classical notion. We will simply write $X$ instead of $(X,\mathcal{O}_X)$ if there is no possible confusion. The following lemma has been proven in \cite{Dav1} and \cite{Rota}, and will be mainly used.  

\begin{lem}(\cite[Theorem 3.6]{Dav1}, \cite[Proposition 8]{Rota})\label{lochyp}
Let $\varphi:R \longrightarrow H$ be a homomorphism of hyperrings. Then, for $\mathfrak{p} \in \Spec H$, $\varphi$ induces the following homomorphism $\varphi_\mathfrak{p}$ of hyperrings:
\[\varphi_\mathfrak{p} :R_\mathfrak{q} \longrightarrow H_\mathfrak{p}, \quad \frac{a}{b} \mapsto \frac{\varphi(a)}{\varphi(b)},\quad \textrm{ where $\mathfrak{q}=\varphi^{-1}(\mathfrak{p})$ } \] 
such that if $\mathfrak{m}_\mathfrak{p},\mathfrak{m}_\mathfrak{q}$ are unique maximal hyperideals of $H_\mathfrak{p}$ and $R_\mathfrak{q}$ respectively, then $\varphi_\mathfrak{p}^{-1}(\mathfrak{m}_\mathfrak{p})=\mathfrak{m}_\mathfrak{q}$.
\end{lem}

\begin{pro}\label{fullyfaithful}
Let $R$ and $H$ be hyperdomains, and $X=\Spec R$, $Y=\Spec H$. Then, we have
\begin{equation}\label{equivofcat}
\Hom(R,H)=\Hom(Y,X),
\end{equation}
where $\Hom(R,H)$ is the set of homomorphisms of hyperrings and $\Hom(Y,X)$ is the set of morphisms of locally hyperringed spaces.
\end{pro}
\begin{proof}
Clearly, a homomorphism $\varphi:R \longrightarrow H$ of hyperdomains induces the continuous map 
\[f:Y=\Spec H \longrightarrow X=\Spec R, \quad \mathfrak{p} \mapsto \varphi^{-1}(\mathfrak{p}).\]
Then, $f$ induces the morphism of sheaves: $f^\#:\mathcal{O}_X \longrightarrow f_{*}\mathcal{O}_Y$. Indeed, for an open subset $V \subseteq X$, we have $f_{*}\mathcal{O}_Y(V):=\mathcal{O}_Y(f^{-1}(V))=\{t\mid t:f^{-1}(V) \longrightarrow \bigsqcup_{\mathfrak{q} \in f^{-1}(V)}H_\mathfrak{q}\}$, where $t$ satisfies the local condition \eqref{sheafcondition}. First, we define 
\begin{equation}\label{psidef}
\psi_V:=\bigsqcup_{\mathfrak{p} \in V}\varphi_\mathfrak{p}:\bigsqcup_{\mathfrak{p}\in f^{-1}(V)}R_{\varphi^{-1}(\mathfrak{p})} \longrightarrow \bigsqcup_{\mathfrak{p} \in f^{-1}(V)}H_\mathfrak{p},
\end{equation}
where $\varphi_\mathfrak{p}$ is the map induced from $\varphi$ at $\mathfrak{p}$ as in Lemma \ref{lochyp}. We also define
\begin{equation}\label{sheafmapdef}
 f^\#(V):\mathcal{O}_X(V) \longrightarrow  \mathcal{O}_Y(f^{-1}(V)), \quad s \mapsto t:=\psi_V \circ s \circ f.
\end{equation}
We need to check four things. Firstly, we have to show that $t$ as in \eqref{sheafmapdef} is an element of $\mathcal{O}_Y(f^{-1}(V))$. Secondly, we have to show that $f^\#$ is compatible with an inclusion $V \hookrightarrow U$ of open sets of $X$; this is clear from the construction. Thirdly, we have to show that $f^\#(V)$ is a homomorphism of hyperrings. The proof of these three is similar to the classical case. Finally, we have to show that $f^\#(V)$ is local. It easily follows from the third statement of Theorem \ref{sheaf} and Lemma \ref{lochyp}.\\
Conversely, suppose that a morphism 
\[(g,g^\#):Y=(\Spec H,\mathcal{O}_Y) \longrightarrow X=(\Spec R,\mathcal{O}_X)\] 
of integral affine hyperring schemes is given. Since $R$ and $H$ are hyperdomains, we can recover a homomorphism of hyperrings $\varphi: R \longrightarrow H$ by taking global sections thanks to Theorem \ref{sheaf}. Therefore, all we have to prove is that the map $(f,f^\#)$ induced from $\varphi$ as in \eqref{sheafmapdef} is same as $(g,g^\#)$. The proof is similar to the classical case. 
\end{proof}



\subsection{From hyperring schemes to the classical schemes}

\noindent Next, we provide an example showing that an integral hyperring scheme can be linked to the classical theory. Let $A$ be an integral domain containing the field $\mathbb{Q}$ of rational numbers, $X=\Spec A$, and $Y=\Spec(A/\mathbb{Q}^{\times})=\Spec (A\otimes_{\mathbb{Z}}\mathbf{K})$ (see Proposition \ref{connect} for the definition of $A/\mathbb{Q}^\times$ or \cite{con3} for the scalar extension functor $-\otimes_{\mathbb{Z}}\mathbf{K}$). We prove that there exists a canonical homeomorphism $\varphi:Y \longrightarrow X$ such that $\mathcal{O}_Y(\varphi^{-1}(U))\simeq \mathcal{O}_X(U) \otimes_{\mathbb{Z}}\mathbf{K}$ for an open subset $U\subseteq X$. Indeed, such homeomorphism is very predictable from the following observation: let $B$ an integral domain containing the field $\mathbb{Q}$ of rational numbers. Then, a polynomial $f \in B[X_1,...,X_n]$ vanishes if and only if $qf$ vanishes $\forall q \in \mathbb{Q}^{\times}$.

\begin{lem}\label{qlocalem}
Let $A$ be an integral domain containing the field $\mathbb{Q}$ of rational numbers. Let $A \ni f\neq 0$ and $\tilde{f}$ be the image of $f$ under the canonical projection map $\pi:A \longrightarrow R:=A/\mathbb{Q}^{\times}$. Then, we have
\[A_f /\mathbb{Q}^{\times} \simeq R_{\tilde{f}},\]
where $R_{\tilde{f}}$ is the localization of $R$ at $\tilde{f}$. 
\end{lem}

\begin{proof}
Since $A$ is an integral domain, $A_f$ contains $A$ (hence, contains $\mathbb{Q}$). Thus $A_f/\mathbb{Q}^{\times}$ is well defined. Let us define the following map:
\[\psi:A_f/\mathbb{Q}^{\times} \longrightarrow R_{\tilde{f}}, \quad [\frac{a}{f^n}] \mapsto \frac{[a]}{\tilde{f}^n},\]
where $[\frac{a}{f^n}]$ is the equivalence class of $\frac{a}{f^n} \in A_f$ in $A_f/\mathbb{Q}^{\times}$ and $[a]$ is the equivalence class of $a \in A$ in $R=A/\mathbb{Q}^{\times}$. One can easily show that $\psi$ is a well-defined and strict homomorphism of hyperrings which is also bijective. It follows from the first isomorphism theorem of hyperrings (cf. \cite[Proposition $2.11$]{Dav1}) that $\psi$ is an isomorphism of hyperrings. 
\end{proof}

\begin{lem}\label{qfraclem}
Let $A$ be an integral domain containing the field $\mathbb{Q}$ of rational numbers and $R=A/\mathbb{Q}^{\times}$. Then, we have
\[\Frac(A)/\mathbb{Q}^{\times} \simeq \Frac(R)=\Frac(A/\mathbb{Q}^{\times}).\]
\end{lem}

\begin{proof}
The proof is similar to Lemma \ref{qlocalem}. For the notational convenience, let us define the following map:
\[\psi:\Frac(A)/\mathbb{Q}^{\times} \longrightarrow \Frac(R), \quad [\frac{b}{a}] \mapsto \frac{[b]}{[a]}.\]
Again, we have to show that this is well defined, bijective, and a strict homomorphism of hyperrings. The proof is almost identical to that of Lemma \ref{qlocalem}.
\end{proof}

\begin{lem}\label{qutientanddomainhomeo}
Let $A$ be a commutative ring, $G\subseteq A^{\times}$ be a multiplicative subgroup, and $A/G$ be the quotient hyperring. Then, $X=\Spec A$ and $Y=\Spec(A/G)$ are homeomorphic (under the Zariski topology).
\end{lem}\label{homeo}
\begin{proof}
If $G=\{1\}$ then there is nothing to prove. Thus, we may assume that $|G|\geq 2$. We define the following map:
\[ \sim : X \longrightarrow Y, \quad \mathfrak{q} \mapsto \tilde{\mathfrak{q}}:=\{\alpha G \mid \alpha \in \mathfrak{q}\}.\]
We claim that the map $\sim$ is well defined. Indeed, we have
\[\alpha G = \beta G \Longleftrightarrow \alpha=\beta u, u\in G \subseteq A^{\times} \Longleftrightarrow \alpha,\beta \in \mathfrak{q} \textrm{ or } \alpha, \beta \not \in \mathfrak{q},\]
therefore $\tilde{\mathfrak{q}}$ is uniquely determined by $\mathfrak{q}$. Furthermore, $\tilde{\mathfrak{q}}$ is a hyperideal. In fact, we have $0G \in \tilde{\mathfrak{q}}$. If $aG \in \tilde{\mathfrak{q}}$ then $(-a)G=-(aG) \in \tilde{\mathfrak{q}}$. For $rG \in A/G$ and $aG \in \tilde{\mathfrak{q}}$, since $(rG)(aG)=raG$ and $ra \in \mathfrak{q}$, it follows that $(rG)(aG) \in \tilde{\mathfrak{q}}$. Suppose that $aG,bG \in \tilde{\mathfrak{q}}$. One can observe that $aG, bG \in \tilde{\mathfrak{q}} \Longleftrightarrow a, b \in \mathfrak{q}$ since $G\subseteq A^\times$ and $\mathfrak{q}$ is a prime ideal. Therefore, for $zG \in \tilde{\mathfrak{q}}$, we may assume that $z=at+bh$ for some $t,h \in G$. It follows that $z \in \mathfrak{q}$, hence $zG \in \tilde{\mathfrak{q}}$. This shows that $\tilde{\mathfrak{q}}$ is a hyperideal. Next, we show that $\tilde{\mathfrak{q}}$ is prime. Suppose that $(aG)(bG)=(abG) \in \tilde{\mathfrak{q}}$ and $aG \not \in \tilde{\mathfrak{q}}$. This implies that $ab \in \mathfrak{q}$ and $au \not\in \mathfrak{q}$ $\forall u \in G$, hence $a \not \in \mathfrak{q}$. Since $\mathfrak{q}$ is prime, this implies that $b \in \mathfrak{q}$, and $bG \in \tilde{\mathfrak{q}}$.\\
Next, we claim that the map $\sim$ is continuous. Let $\varphi:=\sim$ for the notational convenience. It is enough to show that $\varphi^{-1}(D(fG))$ is open. We have the following:
\begin{equation}\label{homeoeq}
\varphi^{-1}(D(fG))=D(f).
\end{equation}
Indeed, if $\mathfrak{q} \in D(f)$, then $\varphi(\mathfrak{q})=\tilde{\mathfrak{q}}$ can not contain $fG$ by definition. Hence, $D(f) \subseteq \varphi^{-1}(D(fG))$. Conversely, suppose that $\mathfrak{p} \in \varphi^{-1}(D(fG))$, then $\varphi(\mathfrak{p}) \in D(fG)$. Since $\varphi(\mathfrak{p})=\tilde{\mathfrak{p}}=\{\alpha G \mid \alpha \in \mathfrak{p}\}$ and $f \not \in \mathfrak{p}$, it follows that $\mathfrak{p} \in D(f)$, hence $\varphi^{-1}(D(fG)) \subseteq D(f)$. This proves \eqref{homeoeq}, hence $\sim$ is continuous.\\
Finally, we construct the inverse of the map $\varphi=\sim$. The canonical projection map $\pi:A\longrightarrow A/G$ induces the following canonical map:
\[\psi:Y \longrightarrow X, \quad \mathfrak{p} \mapsto \pi^{-1}(\mathfrak{p}).\]
Clearly, $\psi$ is continuous since $\psi^{-1}(D(f))=D(fG)$. We claim that $\varphi$ and $\psi$ are inverses to each other. Since both $\varphi$ and $\psi$ are continuous, it is enough to show that $\varphi$ is bijective and $\varphi\circ \psi=id_Y$. First, we show that $\varphi$ is injective. Assume that $\varphi(\mathfrak{q})=\varphi(\mathfrak{p})$ for $\mathfrak{p},\mathfrak{q} \in X$. Then, for $x \in \mathfrak{q}$, we have $y \in \mathfrak{p}$ such that $xG=yG$. It follows that $x=yg$ for some $g \in G$, hence $x \in \mathfrak{p}$. Since the argument is symmetric, we have $\mathfrak{p}=\mathfrak{q}$.\\ 
For the surjectivity of $\varphi$, take an element $\wp \in \Spec A/G$. We consider $\alpha G$ as the subset $\alpha G:=\{\alpha g \mid g \in G\}\subseteq A$ and define 
\[\mathfrak{p}:=\bigcup_{\alpha G \in \wp}\alpha G.  \]
We have to show that $\mathfrak{p}$ is a prime ideal of $A$. We have $0 \in \mathfrak{p}$. Moreover, $a \in \mathfrak{p} \Longleftrightarrow a \in \alpha G$ for some $\alpha G \in \wp$. It follows that $-\alpha G \in \wp$ and hence $-a \in \mathfrak{p}$.
Furthermore, for $a \in \mathfrak{p}$ and $r \in A$, we have $aG \in \wp$ and $rG \in A/G$. It follows from $(rG)(aG)=(raG) \in \wp$ that $ra \in \mathfrak{p}$. If $a,b \in \mathfrak{p}$ then $aG,bG \in \wp$. This implies that $aG+bG \subseteq \wp$ and hence $a+b \in \mathfrak{p}$. This proves that $\mathfrak{p}$ is an ideal. We observe that $\mathfrak{p}$ can not be $A$ since that implies $1 \in \mathfrak{p}$ and $1G \in \wp$, but $\wp\neq A/G$. One further observes that $\mathfrak{p}$ is prime since for $ab \in \mathfrak{p}$ and $a \not \in \mathfrak{p}$, we have $(aG)(bG) \in \wp$ and $aG \not \in \wp$. This implies that $bG \in \wp$, hence $b \in \mathfrak{p}$. Obviously, we have $\varphi(\mathfrak{p})=\wp$. This shows that $\varphi$ is surjective. In fact, one can see that $\mathfrak{p}=\psi(\wp)$. Thus, we have $\varphi(\mathfrak{p})=\varphi \circ \psi(\wp)=\wp$ and therefore $\varphi \circ \psi = id_Y$. This completes our proof.
\end{proof}

\begin{pro}\label{qsheaflem}
Let $A$ be an integral domain containing the field $\mathbb{Q}$ of rational numbers. Let $R:=A/\mathbb{Q}^{\times}$, $X=(\Spec A,\mathcal{O}_X)$, $Y=(\Spec R,\mathcal{O}_Y)$, and $\pi:A \longrightarrow A/\mathbb{Q}^{\times}$ be the canonical projection map. Then, the following hold.
\begin{enumerate}
\item
\[\varphi:\Spec R \longrightarrow \Spec A,\quad \mathfrak{p} \mapsto \pi^{-1}(\mathfrak{p})\]
is a homeomorphism.
\item
For an open subset $U \subseteq X$, we have
\[\mathcal{O}_Y(\varphi^{-1}(U)) \simeq \mathcal{O}_X(U) /\mathbb{Q}^{\times}.\] 
\end{enumerate}
\end{pro}
\begin{proof}
The first assertion is proved in Lemma \ref{qutientanddomainhomeo}. Note that the map induced from the canonical projection $\pi:A \longrightarrow A/\mathbb{Q}^\times$ is the desired homeomorphism $\varphi$.\\ 
For the second claim, we use the following classical identification: 
\begin{equation}\label{classicalidentification}
\mathcal{O}_X(U)=\bigcap_{D(f) \subseteq U} A_f \subseteq K:=\Frac(A).
\end{equation}
Each $\mathcal{O}_X(U)$ is an integral domain containing $\mathbb{Q}$, hence $\mathcal{O}_X(U) /\mathbb{Q}^{\times}$ is well defined. From \eqref{classicalidentification}, we may assume that 
\begin{equation}\label{use1}
\mathcal{O}_X(U)/\mathbb{Q}^{\times} \textrm{ and }  A_f/\mathbb{Q}^{\times} \textrm{ are subsets of } K/ \mathbb{Q}^{\times}.
\end{equation} 
Also, from Lemma \ref{qlocalem} and \ref{qfraclem}, we have 
\begin{equation}\label{use2}
A_f/\mathbb{Q}^{\times} \simeq R_{\tilde{f}}, \quad K/\mathbb{Q}^{\times} \simeq L:=\Frac(R). 
\end{equation}
It follows from \eqref{use1} and \eqref{use2} that
\begin{equation}\label{use1and2}
\mathcal{O}_X(U)/\mathbb{Q}^{\times}=(\bigcap_{D(f) \subseteq U}A_f)/\mathbb{Q}^{\times}=\bigcap_{D(f)\subseteq U} (A_f /\mathbb{Q}^{\times}).
\end{equation}
In fact, the first equality simply follows from \eqref{classicalidentification}. It remains to show the second equality. Indeed, we know that $[a] \in \mathcal{O}_X(U) /\mathbb{Q}^{\times} \Longleftrightarrow qa \in \mathcal{O}_X(U)=\bigcap_{D(f) \subseteq U} A_f \Longleftrightarrow qa \in A_f$ $\forall f \in A$ such that $D(f)\subseteq U$ for some $q \in \mathbb{Q}^{\times}$. This implies that $[a] \in A_f /\mathbb{Q}^{\times}$ $\forall f \in A$ such that $D(f) \subseteq U$. It follows that
\[\mathcal{O}_X(U)/\mathbb{Q}^{\times}=(\bigcap_{D(f) \subseteq U}A_f)/\mathbb{Q}^{\times} \subseteq \bigcap_{D(f)\subseteq U} (A_f /\mathbb{Q}^{\times}).\]
Conversely, if $[a] \in \bigcap_{D(f)\subseteq U} (A_f /\mathbb{Q}^{\times})$ then $[a] \in A_f/\mathbb{Q}^{\times}$ $\forall f \in A$ such that $D(f)\subseteq U$. In other words, for each $f$, there exists $q_f \in \mathbb{Q}^{\times}$ such that $aq_f \in A_f$. However, $\mathbb{Q}\subseteq A_f$ $\forall f \in A$, hence $a \in A_f$ and $a \in \bigcap_{D(f)\subseteq U} A_f$. It follows that
\[ [a] \in (\bigcap_{D(f)\subseteq U}A_f)/\mathbb{Q}^{\times}=\mathcal{O}_X(U)/\mathbb{Q}^{\times}, \]
and this shows \eqref{use1and2}.\\
Finally, let $\tilde{f}=\pi(f)$ as in Lemma \ref{qlocalem}. Then, there is a one-to-one correspondence between the following sets (cf. the proof of Lemma \ref{qutientanddomainhomeo}): 
\begin{equation}\label{twoequal}
\mathcal{A}:=\{f \in A \mid D(f) \subseteq U\},\quad \mathcal{B}:=\{\tilde{f} \in R \mid D(\tilde{f}) \subseteq \varphi^{-1}(U)\},
\end{equation}
where $\varphi$ is the canonical homeomorphism in the first assertion. Therefore, together with Theorem \ref{sheaf}, we have
\[\mathcal{O}_X(U)/\mathbb{Q}^{\times}=\bigcap_{D(f)\subseteq U}(A_f/\mathbb{Q}^{\times}) \simeq \bigcap_{D(f)\subseteq U} R_{\tilde{f}} = \bigcap_{D(\tilde{f})\subseteq \varphi^{-1}(U)} R_{\tilde{f}} \simeq \mathcal{O}_Y(\varphi^{-1}(U)).\]
This proves the second assertion.
\end{proof}
\begin{rmk}
Proposition \ref{qsheaflem} states that if $A$ is an integral domain containing $\mathbb{Q}$, then $X:=\Spec A \simeq X_{\mathbf{K}}=\Spec A_\mathbf{K}=\Spec(A/\mathbb{Q}^{\times})$. In other words, the spaces are homeomorphic, but their functions(sections) are different. In fact, what the second assertion states is that sections of $X_{\mathbf{K}}$ can be derived from sections of $X$ by tensoring them with $\mathbf{K}$ in the sense of \cite{con3}.
\end{rmk}

\appendix 
\section{From semistructures to hyperstructures}\label{app}
\subsection{Basic definitions of semiring theory}
\begin{mydef}
A semigroup is a set $M$ with a binary operation $*$ such that 
\[a*(b*c)=(a*b)*c \quad \forall a,b,c \in M.\]
If $a*b=b*a$ $\forall a,b \in M$, we say that $M$ is a commutative semigroup. If there exists a unique element $1_M \in M$ such that $m*1_M=1_M*m=m$ $\forall m \in M$, then $M$ is called a monoid.  
\end{mydef}
\begin{mydef}
A semiring is a set $S$ with two binary operations $+$ and $*$ such that $(S,+)$ is a commutative monoid with the identity element $0_S$ and $(S,*)$ is a monoid with the identity element $1_S$ such that
\begin{enumerate}
\item
$a*(b+c)=a*b+a*c$, $(b+c)*a=b*a+c*a$, $\forall a,b,c \in S.$
\item
$0*a=0$ $\forall a \in S$.
\end{enumerate}
If $(S,*)$ is a commutative monoid then $S$ is called a commutative semiring. If $(S\backslash\{0\},*)$ is an abelian group then $S$ is called a semifield.
\end{mydef}
In what follows we always assume that all semirings and monoids are commutative. 

\begin{mydef}
\begin{enumerate}
\item
By an idempotent monoid we mean a monoid $M$ such that $a+a=a$ $\forall a \in M$. An (additively) idempotent semiring is a semiring $(S,+,*)$ such that $(S,+)$ is an idempotent monoid. 
\item
An idempotent monoid (or a semiring) $M$ is called selective if $x+y \in \{x,y\}$ $\forall x,y \in M$. 
\end{enumerate}

\end{mydef}
\begin{myeg}
Let $\mathbb{B}:=\{0,1\}$ be a two point set. We impose the commutative multiplication as $1*0=0$ and $1*1=1$. The commutative addition is given by $1+1=1$ and $1+0=1$. $\mathbb{B}$ is the smallest idempotent semifield called the boolean semifield.
\end{myeg}
Let $M$ be a monoid with the identity element $0$. One can define the following canonical partial order on $M$:
\begin{equation}\label{totalorder}
x \leq y \Longleftrightarrow x+y=y.
\end{equation} 
By a partial order on $M$ we mean a binary relation on $M$ which is reflexive, transitive, and antisymmetric. A partial order is said to be total if for any $x,y \in M$, we have $x\leq y$ or $y \leq x$. We claim that when $M$ is selective, such order is total. In fact, we know that $x+y=x$ or $x+y=y$ $\forall x,y \in M$, hence $x\leq y$ or $y \leq x$. For an introduction to theory of semirings we refer the readers to \cite{semibook}.
\subsection{The symmetrization process} \label{sym}
In his paper \cite{Henry}, S.Henry introduced the symmetrization process which generalizes in a suitable way the construction of the Grothendieck group completion of a multiplicative monoid. This process allows one to encode the structure of a $\mathbb{B}$-semimodule as the `positive' part of a hypergroup interpreted as a `$\mathbf{S}$-hypermodule'.\\ 
We briefly recall this symmetrization process. Suppose that $M$ is a totally ordered idempotent monoid. We introduce the following notation
\[s(M):=\{(s,1),(s,-1),0=(0,1)=(0,-1) \mid s \in M\backslash \{0\}\}.\]
To minimize the notation we denote $(s,1):=s$, $(s,-1):=-s$, and $|(s,1))|=|(s,-1)|=s$. For any $X=(x,p) \in s(M)$, we define sign$(X)=p$. Then $s(M)$ is a hypergroup with the addition given by
\begin{equation}\label{add}
x+y = \left\{ \begin{array}{ll}
x & \textrm{if $|x|>|y|$ or $x=y$}\\
y & \textrm{if $|x|<|y|$ or $x=y$}\\
$[$-x,x$]$=\{(t,\pm 1)\mid t\leq |x|)\} & \textrm{if $y=-x$}
\end{array} \right.
\end{equation}
We denote with $s:M \longrightarrow s(M)$, $s \mapsto (s,1)$ the associated map.\\
Let $H$ be a hypergroup and $M$ be a monoid. We say that a map $f:M \longrightarrow H$ is additive if 
\[f(0)=0 \textrm{ and } f(a+b) \in f(a)+f(b) \subseteq H\quad  \forall a,b, \in M.\] 
Henry proved that the construction of $s(M)$ determines the minimal hypergroup associated to a commutative monoid $M$ as the following universal property states.\\

\textbf{(Universal Property)}:
Let $M$ be a monoid such that the canonical order as in \eqref{totalorder} is total. Then, for any hypergroup $K$ and an additive map $g:M \longrightarrow K$, there exists a unique homomorphism $h:s(M) \longrightarrow K$ of hypergroups such that $g=h \circ s$.\\
In fact, such $h:s(M) \longrightarrow K$ is given by $h(X)=\textrm{sign}(X)g(x)$,  $\forall X=(x,p) \in s(M)$ (cf. \cite[Theorem $5.1$]{Henry}).

\begin{rmk}\label{upgradereamk}
\begin{enumerate}
\item
Let $M$ be a monoid such that the canonical order as in \eqref{totalorder} is total. Assume also that $M$ is equipped with a smallest element. Then $M$ can be upgraded to a semiring by defining the addition law as the maximum (with respect to the canonical order) and the multiplication as the usual addition.
For example, $\mathbb{R}_{max}$ is the semifield obtained from the (multiplicative) monoid $(\mathbb{R} \cup \{-\infty\},+)$. 
\item
The symmetrization process can be applied to a general class of monoids (cf. \cite[Theorem $5.1 $]{Henry}). In fact, for a monoid $M$, $s(M)$ is a hypergroup if and only if $M$ satisfies the following condition: for all $x,y,z,w \in M$,
\begin{equation}\label{symconditiongeneral}
x+y=z+w \Longrightarrow \exists b \in M\textrm{ s.t. } 
\left\{ \begin{array}{ll}
x+b=z;\quad b+w=y,\\
\textrm{or }x=z+b;\quad  w=b+y
\end{array} \right.
\end{equation}
Henry also proved (cf. \cite[Proposition $6.2$]{Henry}) that when $M$ is an idempotent monoid, $M$ fulfills the condition \eqref{symconditiongeneral} if and only if the canonical order of $M$ as in \eqref{totalorder} is total.
\item
As Connes and Consani pointed out in \cite{con7}, the symmetrization process can be understood in terms of the functor `extension of scalars'.
\end{enumerate}
\end{rmk}

If $B=(B,+,\cdot)$ is a semiring such that the symmetrization process can be applied to the additive monoid $(B,+)$, the multiplicative structure of $B$ induces the corresponding multiplicative structure on $s(B)$ in the component-wise way. In other words, one can define the multiplication law on $s(B)$ such that:
\[(x,p)\cdot (y,q):=(xy,pq),\quad p,q \in \{-1,1\};\quad  1\cdot 1=(-1)\cdot(-1)=1,\quad 1\cdot(-1)=-1.\] 

\begin{rmk}\label{symmetonlymulti}
Let $M$ be a semiring allowing for the symmetrization process. We will prove that under the component-wise multiplication, $s(M)$ is not a hyperring but only a multiring (cf. \cite{mars1}). A multiring is a weaker version of a hyperring in the sense that a hyperring fulfills the distributive law $x(y+z)=xy+xz$ whereas the notion of a multiring only assumes the weak distributive property
\[ x(y+z) \subseteq xy+xz.\]
For example, let $M$ be the semiring whose underlying set is $\mathbb{Z}_{\geq 0}$ with the addition given by $x+y:=\max\{x,y\}$, and the multiplication given by the usual multiplication. Then $s(M)$ does not satisfy the distributive law. For example, $2(3-3) \neq 6-6=[-6,6]$. Indeed, we have $5 \in 6-6=[-6,6]$, but $5$ can not be an element of $2(3-3)=2\cdot[-3,3]$ because $2$ can not divide $5$ in usual sense. For $s(M)$ to satisfy the distributive law it seems necessary to add a suitable divisibility condition on the multiplication of $s(M)$. A particular case is studied in \cite{jaiungthesis}.
\end{rmk}

\begin{lem}\label{s(B)multi}
Let $B$ be a selective semiring. Then the symmetrization $B_S:=s(B)$ is a multiring with the component-wise multiplication. In particular, if $B$ is a totally ordered (with the canonical order) semifield, then $B_S$ is a hyperfield.
\end{lem}
\begin{proof}
We first note that $x\leq y$ implies $xz \leq yz$ for all $z \in B$. In fact, it follows from $x\leq y \Longleftrightarrow x+y=y$ that $xz+yz=yz \Longleftrightarrow xz\leq yz$. For the first assertion, all we have to show is that $X(Y+Z)\subseteq XY+XZ$ for all $X,Y,Z \in B_S$. If $X=0$, then there is nothing to prove. Therefore we may assume that $X\neq 0$. Let $Y=(y,p)$, $Z=(z,q)$, $X=(x,r)$. When $\#(Y+Z)=1$, it follows from \eqref{add} that there are three possible cases. The first case is when $Y=Z$. In this case, we have $X(Y+Z)=XY=XY+XY=XY+XZ$. The second case is when $p=q$, but $y\neq z$. Since $B$ is selective, we may further assume that $y>z$. Therefore, we have $xy \geq xz$ and $X(Y+Z)=XY \in XY+XZ$. The final case is when $p \neq q$ and $y \neq z$. But, in this case, the similar argument as the second case shows that $X(Y+Z)\subseteq XY+XZ$. When $\#(Y+Z) \neq 1$, from \eqref{add}, we may assume that $Y=(y,1)$, $Z=(y,-1)$, and $X=(x,r)$. Take any $T=(t,p) \in (Y+Z)$. It follows from \eqref{add} that $t \leq y$, hence $xt \leq xy$. Therefore we have $XT=(xt,pr) \in XY+XZ$. When $B$ is a semifield, each non-zero element of $B_S$ has a multiplicative inverse. Therefore $B_S$ is a multifield and it is well-known that any multifield is a hyperfield (and vice versa).
\end{proof}

\bibliography{Algeo_Hyper}\bibliographystyle{plain}

\end{document}